\documentclass[11pt,letterpaper]{amsart}

\usepackage{amssymb}
\usepackage{amsthm}
\usepackage{epsfig} 
\usepackage{epic,eepic}
\usepackage{xcolor}
\usepackage{comment}
\usepackage{enumitem}
\usepackage{geometry}

\newtheorem{thm}{Theorem}[section]
\newtheorem{prop}[thm]{Proposition}

\newtheorem{cor}[thm]{Corollary}

\theoremstyle{definition}
\newtheorem{definition}[thm]{Definition}

\theoremstyle{remark}
\newtheorem{remark}[thm]{Remark}

\usepackage{hyperref}

\usepackage{esint}

\newcommand{\D}{\mathcal{D}}

\renewcommand{\S}{\mathcal{S}}
\newcommand{\R}{\mathbb{R}}
\newcommand{\p}{\partial}

\renewcommand{\O}{\Omega}
\newcommand{\spt}{\operatorname{spt}}

\newcommand{\eps}{\epsilon}

\newcommand{\sgn}{\operatorname{sgn}}

\renewcommand{\tt}[1]{\text{#1}}

\numberwithin{equation}{section}

\begin{document}

\title{$L^p$ stability of Vortex Patches in Two Dimensional Domains}

 \author{Zelin Dong}
 \address{Department of Mathematics, The Chinese University of Hong Kong, 
 Shatin, NT, Hong Kong SAR}
 \email{1155173731@link.cuhk.edu.hk}

\flushbottom

\begin{abstract}
In this paper, we investigate the orbital stability of vortex patches in the two-dimensional incompressible Euler equations, extending the penalized energy variational framework pioneered by Abe and Choi \cite{abe2022stability} for Lamb dipoles. The recent work by Abe, Choi and Jeong \cite{Abe2025StabilityOL} (which removes $L^1$ constraint) and Dong and Luo \cite{Dong2026StabilityOV} (which treats domains lacking scaling or translation invariance) left open the challenge of a unified $L^p$ stability theory without any a priori $L^1$ or $L^p$ bounds on two-dimensional domains.

We establish a unified $L^p$ stability theory on three typical two-dimensional domains. These domains are: the half-plane, strips of any width, and domains satisfying a weak finite volume condition. For each domain, we prove that the penalized energy functional admits a minimizer for suitable $p$, and that every such minimizer satisfies the elliptic equation $\omega^{p-1} = \lambda(\psi - W x_2)_+$. Furthermore, we demonstrate that the set of minimizers is orbitally stable under the Eulerian dynamics. The absence of spatial scaling and horizontal translation invariance necessitates novel strategies: on the strip, we refine a concentration-compactness argument to prove strict subadditivity; on weak finite volume domains, we bypass the need for subadditivity by exploiting the inherent decay rate $q$ of the domain to enforce compactness. 
This work synthesizes the approaches of \cite{abe2022stability}, \cite{Abe2025StabilityOL}, \cite{abe2025existence}, and \cite{Dong2026StabilityOV} into a comprehensive framework, significantly expanding the scope of provably stable vortex structures.
\end{abstract}

\maketitle

\tableofcontents

\thispagestyle{empty}

\section{Introduction}

\subsection{Background and Motivation}

The two-dimensional Euler equations for the motion of an inviscid, incompressible fluid on a fixed domain $\Omega\subseteq \R^2_+$ is given in vorticity form by
\begin{align}
    \p_t \omega +u\cdot\nabla \omega&= 0, ~ \quad \omega(x,0)=\omega_0(x) \quad \tt{ in } \Omega.  \label{eqn:vorticity equation}
\end{align}
Here, $\omega$ is the vorticity of the flow, and the fluid velocity $u$ satisfies the slip boundary condition: 
\begin{align*}
    u\cdot \mathcal{N}=0,
\end{align*}
where \(\mathcal{N}\) is the unit outward normal. When the appropriate Biot-Savart law is applicable, $u$ can be recovered by from $\omega$ in the sense that $$u(x,t)=\nabla^\perp \int_D G(x,y)\omega(y,t)dy,$$ and $G$ is the Green's function for the Dirichlet problem in $\Omega$. It will be convenient for us to define the stream function associated with $\omega$ :
\begin{align}
    \psi(x,t)=\int_{\Omega} G(x,y) \omega(y,t)dy, \label{def:stream function}
\end{align}
so that $u(x,t)=\nabla^\perp \psi$. In addition, in this paper, when conducting a discussion at a fixed time, the time variable $t$ will be omitted, i.e., $\omega (\cdot, t) = \omega (\cdot), \psi (\cdot, t) = \psi (\cdot) $ and $u(\cdot, t) = u(\cdot) $.

A natural question is whether solutions remain stable for long times. In general, the answer is negative. For 2D Euler equations in a disk, Nadirashvili, Alexander, and Šverák \cite{kiselev2014small} constructed solutions whose vorticity gradient grows double‑exponentially. Similar mechanisms were exploited on the torus $\mathbb{T}^2$ by Zlato{\v{s}} \cite{zlatovs2015exponential}, applied to smooth domains with an axis of symmetry by Xu \cite{xu2016fast}, and also for free-boundary problems by Hu, Luo, and Yao \cite{hu2024small}, demonstrating that instability can occur generically.

Nevertheless, stability can be established for certain domains with suitable classes of initial data. Cao and Wang \cite{cao2021nonlinear} proved nonlinear stability of highly concentrated vortex patches near non‑degenerate minima of the Robin function; Cao, Wan, and Wang \cite{cao2019nonlinear} showed orbital stability for kinetic energy maximizing patches; Choi, Jeong, and Lim \cite{choi2022stability} studied the stability of monotone, nonnegative, and compactly supported vorticities in a half cylinder. Earlier work by Burton \cite{burton2005global} and Turkington \cite{turkington1987evolution} established foundational stability results for vortex patches using rearrangement techniques. 


In recent years, a variational approach based on minimizing a penalized energy has been successfully applied to prove the orbital stability of vortex patches. Abe and Choi \cite{abe2022stability} established the stability of Lamb dipoles in the half-plane $\R_+^2$ by considering the minimization problem
\begin{align*}
    \widetilde{I}_{\mu,\lambda}(2)=\inf_{\omega \in \widetilde{K}_{\mu,\nu}}\{-E_{2,\lambda}[\omega]\}, 
\end{align*}
where $\widetilde{K}_{\mu,\nu}$ is a class of admissible vorticities with fixed first moment and bounded $L^1$-norm, and $E_{2,\lambda}$ is the penalized energy. In recent years, their variational approach has been extended to many further problems. Choi, Jeong, and Yao \cite{choi2024stability} used the monotonicity of the first moment and pointwise kernel estimates to study the orbital stability of a pair of opposite-signed Lamb dipoles under odd-odd symmetry, as well as concentrated vortices in a quadrant. Abe, Choi, and Jeong \cite{Abe2025StabilityOL} overcame the $L^1$-norm restriction required for stability in Abe and Choi \cite{abe2022stability}. Abe, Choi, Jeong, Sim, and Woo \cite{abe2025existence} focused on problems with $L^p$-norm structure and proved the existence and stability of Sadovskii vortices. Dong and Luo \cite{Dong2026StabilityOV} addressed the limitation of previous studies confined to $\R^2_+$, extending $L^2$ stability results to strips and to a class of domains satisfying the weak finite volume condition.

However, all of these results have their own shortcomings. Abe, Choi, and Jeong \cite{Abe2025StabilityOL} is restricted to the $L^2$ setting; Abe, Choi, Jeong, Sim, and Woo \cite{abe2025existence} studied the stability of maximizers of the kinetic energy, and thus the stability property they obtained differs somewhat from other results; 
Recently, in a joint work with Luo \cite{Dong2026StabilityOV}, we overcame
the domain limitation, but did not achieve $L^p$ stability and still required a prior $L^1$ upper bound.

\subsection{Main results} \label{sec:main results}
The main goal of this paper is to synthesize all the above characteristics and obtain $L^p$ stability results for proper $p$ that are free of both the $L^1$ and $L^p$ constraints in $\O$, where $\O$ represents the three typical domains mentioned above, which are given by
\begin{definition} 
    Let $\O\subseteq\R^2_+$ be a fixed domain. Three particular cases of $\O$ and their Green’s functions are given below:
    \begin{itemize}
        \item The half plane: $\R^2_+=\R^2\cap\{ x_2 >0 \}$, whose Green’s function is 
        \begin{align}
            G_H(x,y)=\frac{1}{4\pi} \ln \left( 1 + \frac{4x_2y_2}{|x-y|^2} \right) \label{def:green's function for half plane}
        \end{align}
        
        \item The strip: $\S=\R \times (0,L), ~ L>0$. whose Green’s function is
        \begin{align}
            G_S(x,y)=&\frac{1}{4\pi}\ln\Big(1+\frac{\sin(\frac{\pi}{L}x_2)\sin(\frac{\pi}{L}y_2)}{\sinh(\frac{\pi}{2L}(x_1-y_1))^2+\sin(\frac{\pi}{2L}(x_2-y_2))^2}\Big). \label{def:green's function for strip}
        \end{align}

        \item The domain $\D = D_u\cup D_l$, 
        where
        \begin{align}
            D_u := \mathcal{D} \cap \{x_2\ge 1\}, && D_l := \D \cap \{x_2<1\}, \label{def:upper and lower part of domain}
        \end{align}
        that satisfies the weak finite volume condition:
        \begin{itemize}
            \item[*] The volume condition,       
                \begin{align}
                    \int_{D_l} x_2^q dx<\infty, \tt{ for some } q\geq 0,\tt{ and } ~ \tt{Vol}(D_u)<\infty. \label{def:weak_finite_volume_domain_assumption}
                \end{align}
            \item[*] The Biot-Savart law is applicable on $\D$.
        \end{itemize}
        whose Green’s function is denoted by $G_\D$, with no explicit formula.
    \end{itemize}
\end{definition}
\begin{remark}
    For convenience, without specifying domain $\O$, we denote its Green’s function by $G$. Since $\O \subseteq \R^2_+$,
    \begin{align}
        0\leq G(x,y)\leq G_H(x,y) \tt{ for any } x,y\in \overline{\Omega}.\label{ineq:comparison between green's function}
    \end{align}
\end{remark}
\begin{remark}
    The class of domains $\D$ includes a wide range of domains, including but not limited to all simply connected domains with (piecewise) smooth boundary discussed in \cite{lacave2019euler}, as well as non-smooth domains verifying the weak tangency condition discussed in \cite{gerard2013two} and \cite{galdi2011introduction}. In addition, this class covers all bounded domains, domains with periodic boundaries, and domains whose boundaries are graphs of some appropriate functions. These are of considerable importance in the study of free-boundary problems. 
\end{remark}

Throughout the paper, we consider the penalized energy minimization problem
\begin{align}
    I_{\mu,\lambda}(p)=\inf_{\omega \in K_{\mu}}\{-E_{p,\lambda}[\omega]\}, \label{def:minimizing problem}
\end{align}
where the penalized energy functional $E_{p,\lambda}$ for any $\lambda>0$ is the difference between kinetic energy $E$ and the $L^p$-norm in the sense that 
\begin{align*}
    E_{p,\lambda}[\omega]=E[\omega]-\frac{1}{p\lambda}\|\omega\|_{p}^p, \tt{ where } E[\omega]=\frac{1}{2}\int_{\Omega}\int_{\Omega}G(x,y)\omega(x)\omega(y)dxdy, 
\end{align*}
and admissible function space $K_{\mu}$ for any $\mu \geq 0$ is given by 
\begin{align*}
    K_{\mu}(p)=\left\{\omega\in L^p(\Omega) \big| \omega\geq0, \int_{\Omega}x_2\omega(x)dx=\mu \right\}, 
\end{align*}

Define $S_{\mu,\lambda}(p)\subseteq K_{\mu}$ to be the set of minimizers of \eqref{def:minimizing problem}. In the sequel, we suppress $p$ and write $K_{\mu}=K_{\mu}(p), \ I_{\mu,\lambda} = I_{\mu,\lambda}(p)$ and $S_{\mu,\lambda} = S_{\mu,\lambda}(p)$ when p is fixed.
Throughout this paper, $C$ denotes general constants that depend on $p$ and the domain $\Omega$. We can now state the existence of minimizers and demonstrate their stability.

\begin{thm}[Existence and compactness of minimizers]\label{thm:general convergence theorem without L1}
    Assume one of the following holds:
    \begin{enumerate}[label=\ref{thm:general convergence theorem without L1}.\arabic*]
        \item \label{case:general convergence on half-plane} $\O=\R^2_+$ with $p>4/3$ and $\mu,\lambda>0$; 
        \item \label{case:general convergence on strips} $\O=\S$ with $p > 2$ and $\mu,\lambda>0$, or $p = 2$ with and $\mu>0$, $\lambda>\frac{2\pi^4}{L^2}$;
        \item \label{case:general convergence on finite volume} $\O=\D$ with $q\in[0,2]$ and $p>4/3$, or $q\in(2,4)$ and $ 4/3< p \leq q/(q-1)$.
    \end{enumerate}
     
    For any minimizing sequence $\{\omega_n\}$ satisfying $\omega_n\in K_{\mu_n}$, $\mu_n \rightarrow \mu $ and $-E_{p,\lambda}[\omega_n] \rightarrow I_{\mu,\lambda}$, then there exists a subsequence, still denoted by $\{\omega_n\}$, such that there exists $\omega\in K_{\mu}$, 
    $\omega_n\rightarrow\omega$ and $x_2\omega_n\rightarrow x_2\omega$ strongly in $L^p(\O)$ and $L^1(\O)$ respectively. In particular, $S_{\mu,\lambda} \neq \emptyset$.
\end{thm}

\begin{thm}\label{thm:stable wrt minimizer}
    Under the same assumption of the Theorem \ref{thm:general convergence theorem without L1}, the $S_{\mu,\lambda}$ is orbitally stable in the sense that: for any $\epsilon>0$, there exists $\delta>0$ such that $\forall ~ \xi_0\in C_c^{\infty}(\Omega), ~ x_2\xi_0\in L^1(\Omega), ~ \xi_0\geq0$ and 
    \begin{align*}
        \inf_{\omega\in S_{\mu,\lambda}(p)}\Big\{ \|\xi_0-\omega\|_{p}+\|x_2(\xi_0-\omega)\|_{1} \Big\}\leq \delta,
    \end{align*}
    the solution $\xi(t)$ satisfy: $\forall ~ t \in [0,T_{\text{lifespan}})$,
    \begin{align}
        \inf_{\omega\in S_{\mu,\lambda}(p)}\Big\{ \|\xi(t)-\omega\|_{p}+\|x_2(\xi(t)-\omega)\|_{1} \Big\}\leq \epsilon. \label{eqn:stable result}
    \end{align}
\end{thm}

\begin{remark}[Comparison with previous work] Theorems \ref{thm:general convergence theorem without L1} and \ref{thm:stable wrt minimizer} unify and extend the following earlier results:
\begin{enumerate}
    \item For $\O=\R^2_+$, we extend the mass-free $L^2$ stability in \cite{Abe2025StabilityOL} to $L^p$ for $p>4/3$, and complement such stability from penalized energy perspective;
    \item For $\O=\S$, we get rid of the smallness assumption of $\mu$ for $L^2$ stability in \cite{Dong2026StabilityOV}, and generalized it to mass-free $L^p$ stability for $p \geq 2$ with an explicit parameter threshold;
    \item For $\O=\D$, we extend the $L^2$ stability in \cite{Dong2026StabilityOV} to $L^p$ based on different $q$ without the $L^1$ constraint.
\end{enumerate}
\end{remark}

It is worth noting that when $\O=\S$, Theorem \ref{thm:stable wrt minimizer} only establishes stability for certain ranges of $p$ and the decay rate $q$; in fact, $\D$ should have additional properties.
Therefore, as a supplement, we impose an $L^1$-norm upper bound on the admissible function space and obtain the following stability for all $q \geq 0$ and $p > 1$,
\begin{thm}\label{thm:stable wrt minimizer with L1 bound}
    Let $\O=\D$, $p, \nu, \mu, \lambda$ given in the Theorem \ref{thm:general convergence for D with L1}, then $\widetilde{S}_{\mu,\nu,\lambda}$ is orbitally stable in the sense that: for any $\epsilon>0$, there exists $\delta>0$ such that $\forall ~ \xi_0\in C_c^{\infty}(\Omega), ~ x_2\xi_0\in L^1(\Omega), ~ \|\xi_0\|_1\leq \nu ~, ~ \xi_0\geq0$ and 
    \begin{align*}
        \inf_{\omega\in \widetilde{S}_{\mu, \nu, \lambda}}\Big\{ \|\xi_0-\omega\|_{p}+\|x_2(\xi_0-\omega)\|_{1} \Big\}\leq \delta,
    \end{align*}
    the solution $\xi(t)$ satisfy: $\forall ~ t \in [0,T_{\text{lifespan}})$,
    \begin{align}
        \inf_{\omega\in \widetilde{S}_{\mu, \nu, \lambda}}\Big\{ \|\xi(t)-\omega\|_{p}+\|x_2(\xi(t)-\omega)\|_{1} \Big\}\leq \epsilon. \label{eqn:stable result with L1}
    \end{align}
\end{thm}
However, since this part involves a modification of the admissible space, its proof is somewhat detached from that of the other theorems; we therefore place its notations and discussion in Appendix \ref{sec:stability for D with L1 bound}.    

\subsection{Major difficulties}
Adapting the proofs of \cite{abe2022stability}, \cite{Abe2025StabilityOL}, \cite{abe2025existence} and \cite{Dong2026StabilityOV} to a general domain $\O \subseteq \R^2_+$ without $L^1$-norm constraint encounters several main difficulties that fundamentally require new tools.

\subsubsection{Losing of natural bounds.} \label{sec:losing natural bound}
In the absence of an $L^1$ or $L^p$ upper bound, minimizing sequences for $-E_{p,\lambda}$ may not be precompact in $L^p$. In \cite{abe2022stability} and \cite{Dong2026StabilityOV}, this was remedied by imposing a priori $L^1$ bound; in \cite{abe2025existence}, the kinetic energy maximization problem is naturally equipped with a prior $L^p$ bound. To remove such a given bound, one should adapt the control of the kinetic energy without an $L^1$ bound as developed in \cite{abe2025existence} to penalized energy functional model. For this functional, when $I_{\mu,\lambda}<0$ and $\frac{2p}{3p-2}<p$, one can use Hölder's inequality to obtain a uniform $L^p$ upper bound for the minimizing sequence and complete the proof. The critical threshold is $p=4/3$.

This is one of the main reasons why our main theorem only includes stability for $p>4/3$. To push beyond the critical point $4/3$, one would need to improve the upper bound provided in by \cite{abe2025existence}. However, the estimate in \cite{abe2025existence} is the optimal value obtained by optimizing the different contributions of the far and near fields of $\omega$ in the integrals over $\mathbb{R}^2_+$, making it rather difficult to reach $4/3$ on $\mathbb{R}^2_+$. As for $\S$ or $D$, the Green's function $G_S$ of $\S$ is difficult to handle, and the Green function $G_\D$ of $\D$ has no explicit expression, making further progress in either case difficult.

\subsubsection{Losing homogeneity.} \label{sec:in homogeneous issue}
Unlike the stability study for $L^2$, the non-homogeneity of $E[\omega]$ and the penalization term $\|\omega\|_p^p$ in $L^p$ setting lead to difficulties with both coercivity and strict subadditivity issues. That is, the following may fail:
\begin{align*}
    I_{\mu,\lambda}<0, \quad I_{\mu,\lambda} < I_{\alpha,\lambda}+I_{\mu-\alpha,\lambda} \tt{ for } 0<\alpha<\mu.
\end{align*}
Losing these structures is a serious obstacle to proving the existence of minimizer via the Lions’ compactness principle. As stated in Section \ref{sec:losing natural bound}, the failure of $I_{\mu,\lambda}<0$ leads to the loss of precompactness, so $p,\mu,\lambda$ needs to be chosen carefully. Besides these parameters, the geometric features of $\O$ also play a key role. Therefore, we adopt the approach in \cite{Dong2026StabilityOV} to obtain a larger feasible parameter range and adjust this range in subsequent research. 

Losing $I_{\mu,\lambda} < I_{\alpha,\lambda}+I_{\mu-\alpha,\lambda} \tt{ for } 0<\alpha<\mu$ is another issue. Without such strict subadditivity, one can not exclude the dichotomy case in the Lions’ compactness principle and thus break the framework. To resolve it, we employ three different strategies for $\R^2_+$, $\S$, and $\D$ respectively. The solution for $\R^2_+$ is scaling, i.e. write $I_{\mu,\lambda}$ to be a multiple of $I_{1,\lambda}$. The methods for $\S$ and $\D$ are explained in following subsections.

\subsubsection{Loss of spatial scaling.}  \label{sec:difficulty_loss_scaling}
As explained in Section 1.1.1 in \cite{Dong2026StabilityOV}, the scaling method is not applicable when $\O=\S$ or $\D$. In strips, it affects the proof of strict subadditivity mentioned in Section \ref{sec:in homogeneous issue} which means that we need a more refined discussion to the behavior of the minimizing sequence. To finish such goal, we develop an alternative route to proving such a property. First, find a suitable range of $(\mu,\lambda)$ such that $I_{\mu,\lambda}<0$ and the lower bound for $\lambda$ is independent from $\mu$. Next, we further investigate the compact support property of minimizers of $-E_{p,\lambda}$ within this parameter range. Then, using Steiner symmetrization, we focus on a class of candidate functions and study the concentration of kinetic energy and its continuity under weak convergence for these functions. With such continuity, the existence of a minimizer is proved. Finally, combining the above materials, the existence of a compactly supported minimizer yields the strict subadditivity of $I_{\mu,\lambda}$.


\subsubsection{Loss of horizontal translation invariance.} 
It is worth noting that the new proof of strict subadditivity for $\S$ relies heavily on the translation invariance, a property that the class of domains $\D$ does not possess. To overcome this difficulty, we construct a new path that avoids the concentration-compactness argument. We exploit the decay of the domain itself to establish compactness. If the decay rate $q$ lies in $[0,2]$, the minimizing sequence is compact for any $p>4/3$. When $q>2$, concentration and stability can only be guaranteed for $p\in (4/3, q/(q-1)]$, which gives no conclusion for $p\geq 4$.

As mentioned in Section \ref{sec:main results}, in Appendix \ref{sec:stability for D with L1 bound}, we demonstrate the $L^p$ stability of any $q \geq 0$ and $p > 1$ under an $L^1$ constraint following the method in \cite{Dong2026StabilityOV} which aims to illustrate further properties of $\D$.

\subsection{Organization of the paper.} The paper is organized as follows. In Section \ref{sec:minimization problem}, we establish the fundamental energy estimates and prove negativity and continuity of $I_{\mu,\lambda}$. Section \ref{sec:study of half plane and strip} is devoted to the half-plane and strip cases, where we prove strict subadditivity (Proposition \ref{prop:strictly subadditivity}) and establish the case \ref{case:general convergence on half-plane} and \ref{case:general convergence on strips} of Theorem \ref{thm:general convergence theorem without L1}. Section \ref{sec:general convergence theorem for D} treats the weak finite-volume domain $\D$, proving the weak continuity of the kinetic energy (Proposition \ref{prop:weak_finite_volume_energy_convergence}) and the case \ref{case:general convergence on finite volume} of Theorem \ref{thm:general convergence theorem without L1}. In Section \ref{sec:stability}, we prove the orbital stability Theorem \ref{thm:stable wrt minimizer}. Section \ref{sec:discussion} concludes with a discussion of open problems and future directions. Appendix \ref{sec:stability for D with L1 bound} contains the stability theory for $\D$ with an 
$L^1$-norm constraint.

\subsection*{Acknowledgement}
The author is very grateful to the supervisor, Professor Chenyun Luo, for the help during the research, and also received support from Hong Kong RGC Grants CUHK--14301225. 

\section{A minimization problem}\label{sec:minimization problem}

We establish the basic properties of the minimization problem. Owing to the observation \eqref{ineq:comparison between green's function}, we derive several useful bounds for the kinetic energy and crossing terms. These are collected in Proposition \ref{prop:estimate for kinetic energy} and are analogous to those obtained in Lemma 2.3 and Lemma 2.4 in \cite{abe2025existence}. Next, we study the sign of $I_{\mu,\lambda}$ in Proposition \ref{prop:propoties of I_mu}, which adopts the idea of Proposition 2.1 in \cite{Dong2026StabilityOV}.

\begin{prop}\label{prop:estimate for kinetic energy}
    The following estimates hold for $p>1$ and nonnegative $\omega,\omega_i\in L^p(\Omega)$, satisfying $x_2\omega, x_2\omega_i\in L^1(\Omega)$ with a constant $C$ independent of $\omega$ and $\omega_i$:
    \begin{align}
        &\int_{\Omega}G(x,y)\omega(y)dy\leq Cx_2^{\frac{p-1}{2p-1}}\|x_2 \omega\|_1^{\frac{p-1}{2p-1}}\|\omega\|_p^{\frac{p}{2p-1}},\label{ineq:estimate of stream fct}\\
        &\int_{\Omega}G(x,y)\omega(y)dy\leq C\|x_2 \omega\|_1^{\frac{2p-2}{3p-2}}\|\omega\|_p^{\frac{p}{3p-2}},\label{ineq:sup estimate of stream fct}\\
        &\int_{\Omega}\int_{\Omega}G(x,y)\omega_1(y)\omega_2(x)dxdy\leq C\|x_2\omega_1\|_1^{\frac{2p-2}{3p-2}}\|\omega_1\|_p^{\frac{p}{3p-2}}\|x_2\omega_2\|_1^{\frac{2p-2}{3p-2}}\|\omega_2\|_p^{\frac{p}{3p-2}},\label{ineq:estimate of energy crossing term}\\
        &E[\omega]\leq C\|x_2\omega\|_1^{\frac{4p-4}{3p-2}}\|\omega\|_p^{\frac{2p}{3p-2}},\label{ineq:estimate of kinetic energy}\\
        &|E[\omega_1]-E[\omega_2]|\leq C\|x_2(\omega_1-\omega_2)\|_1^{\frac{2p-2}{3p-2}}\|\omega_1-\omega_2\|_p^{\frac{p}{3p-2}}\|x_2(\omega_1+\omega_2)\|_1^{\frac{2p-2}{3p-2}}\|\omega_1+\omega_2\|_p^{\frac{p}{3p-2}}.\label{ineq:estimate of energy difference}
    \end{align}
\end{prop}
\begin{proof}
    If $\O=\R^2_+$, it is a combination of Lemma 2.3 and Lemma 2.4 in \cite{abe2025existence}. As for the case $\O=\S$ or $\D$, one can view $L^p(\Omega)=\{\omega\in L^p(\R_+^2) ~ | ~ \spt\omega\subseteq \overline{\O} \}$ and use the inequality $ 0\leq G(x,y)\leq G_H(x,y)$ to obtain
    \begin{align*}
        0<\int_{\O}G(x,y)\omega(y)dy\leq \int_{\R^2_+}G_H(x,y)\omega(y)dy.
    \end{align*}
    All estimates for cases $\O=\S$ and $\D$ then follow directly from the corresponding estimates for $\O=\R^2_+$.
    
\end{proof}

These estimates, while elementary, are fundamental for the rest of the paper. They not only provide uniform bounds for general minimizing sequences (as will be seen in Remark \ref{remark:uniformly bounded}) but also serve as the workhorse for controlling error terms in convergence arguments, such as the proofs of Proposition \ref{prop:strictly subadditivity}, \ref{prop:weak_finite_volume_energy_convergence}, and Theorem \ref{thm:general convergence theorem without L1}. 

Next, we shall prove that $I_{\mu,\lambda}$ is negative for a suitable class of parameters $(\mu,\lambda)$. 

\begin{prop}\label{prop:propoties of I_mu} Recall the definition of $I_{\mu,\lambda}=I_{\mu,\lambda}(p)$ in \eqref{def:minimizing problem}. Let $p>4/3$, then
    \begin{align}
        & I_{0, \lambda}=0, && \text{ for any $\lambda\in \R$}, \label{eqn:I_0}\\
        & I_{\mu, \lambda}>-\infty, && \text{ for any } \mu\geq 0 \text{ and } \lambda>0 ,\label{eqn:I_mu is finite}\\
        & I_{\mu, \lambda}<0, && \text{ for any } (\mu,\lambda)\in \bigcup_{K\subseteq \Omega \text{ and } |K|<\infty}
         (0, +\infty ) \times (\lambda_{K,\mu},\infty ). \label{ineq:negativity of I_mu}
    \end{align}
    where
    \begin{align*}
        \lambda_{K,\mu}=\frac{2 \mu^{p-2}\tt{Vol(K)}}{p \int_K\int_K G(x,y) dydx}\left(\int_K x_2 dx\right)^{2-p}
    \end{align*}
\end{prop}
\noindent\textit{Proof: } The property \eqref{eqn:I_0} is trivial since $K_0=\{0\}$. By \eqref{ineq:estimate of kinetic energy} and Young's inequality, for any $\omega\in K_\mu$,
\begin{align*}
    E_{p,\lambda}[\omega]&\leq C\|x_2\omega\|_1^{\frac{4p-4}{3p-2}}\|\omega\|_p^{\frac{2p}{3p-2}}-\frac{1}{p\lambda}\|\omega\|_p^p\\
    &\leq \left(C\eps^{-1}\|x_2\omega\|_1^{\frac{4p-4}{3p-2}}\right)^{\frac{3p-2}{3p-4}}+\left(\eps^{\frac{3p-2}{2}}-\frac{1}{p\lambda}\right)\|\omega\|_p^p\\
    &\leq C\lambda^{\frac{2}{3p-4}}\mu^{\frac{4p-4}{3p-4}},
\end{align*}
where we choose $\eps=(p\lambda)^{-\frac{2}{3p-2}}$. Then 
\begin{align}
    I_{\mu,\lambda}=\inf_{\omega\in K_{\mu}}\big\{-E_{p,\lambda}[\omega ]\big\}\geq -C\lambda^{\frac{2}{3p-4}}\mu^{\frac{4p-4}{3p-4}}>-\infty.\label{ineq:lowerbound of I_mu,lambda}
\end{align}

For any $ \mu>0 $ and $ \lambda>\lambda_{K,\mu}$ where $K\subseteq \Omega \text{ and } Vol(K)<\infty$, define
\begin{align*}
    \omega_0=c_0 \mathbf{1}_{K}, \text{ where } c_0=\mu\left(\int_{K}x_2dx\right)^{-1}.
\end{align*}
It is easy to see $\omega_0\in K_\mu $ by the choice of $c_0$, and 
\begin{align}
    E_{p,\lambda}[\omega_0]&= \frac{1}{2}c_0^2\int_{K}\int_{K}G(x,y)dxdy-\frac{1}{p\lambda}c_0^p \tt{Vol}(K), \notag\\
    &=\frac{1}{2}c_0^2\left(\int_{K}\int_{K}G(x,y)dxdy-\frac{2c_0^{p-2}}{p\lambda}\tt{Vol}(K)\right), \notag\\
    &>0, \label{eqn:lowerbound of -I_mu}
\end{align}
where the last inequality is positive by the choice of $\lambda$. 

\rightline{$\Box$}

\begin{remark}[The minimizing sequence is bounded in $L^p$]\label{remark:uniformly bounded}
\end{remark}
Assume $p>4/3$. Any minimizing sequence $\{\omega_n\}$ satisfying $\omega_n\in K_{\mu_n}, \mu_n\rightarrow \mu$, and $-E_{p,\lambda}[\omega_n]\rightarrow I_{\mu,\lambda}$ with $\mu, \lambda$ follows the assumption in \eqref{ineq:negativity of I_mu} is uniformly bounded in $L^p$. Indeed, by \eqref{ineq:estimate of stream fct} and Young’s inequality
    \begin{align*}
        \frac{1}{p\lambda}\|\omega\|_{p}^p+E_{p,\lambda}[\omega]
        = E[\omega] \leq C\|x_2\omega\|_1^{\frac{4p-4}{3p-2}}\|\omega\|_p^{\frac{2p}{3p-2}} \leq & C\lambda^{\frac{2}{3p-4}}\|x_2\omega\|_1^{\frac{4p-4}{3p-4}}+\frac{1}{2p\lambda}\|\omega\|_p^p,
    \end{align*}
    then, 
    \begin{align*}
        \|\omega\|_{p}^p\leq C\lambda^{\frac{3p-2}{3p-4}}\|x_2\omega\|_1^{\frac{4p-4}{3p-4}}-2p\lambda E_{p,\lambda}[\omega].
    \end{align*}
    Thus by $I_{\mu,\lambda}<0$, the minimizing sequence satisfies 
    \begin{align}
        \limsup_{n\rightarrow +\infty}\|\omega_n\|_{p}\leq C\lambda^{\frac{3p-2}{3p^2-4p}}\mu^{\frac{4p-4}{3p^2-4p}} + 2p\lambda I_{\mu,\lambda} < C\lambda^{\frac{3p-2}{3p^2-4p}}\mu^{\frac{4p-4}{3p^2-4p}}. 
        \label{ineq:minimizing seq uniformly bounded}
    \end{align}
    In particular, if $\omega$ is a minimizer, then 
    \begin{align}
        \|\omega\|_{p}\leq C\lambda^{\frac{3p-2}{3p^2-4p}}\mu^{\frac{4p-4}{3p^2-4p}}. \label{ineq:uniformly bounded}
    \end{align}

\begin{remark}[Behavior of $\mu,\lambda$]\label{remark:Behavior of mu, lambda}
\end{remark}
As mentioned in Section \ref{sec:losing natural bound}, $I_{\mu,\lambda}$ is not negative for all $p>1$, $\mu,\lambda>0$. This depends not only on the choice of $p$ but, more crucially, on the domain's geometric features. In Section \ref{sec:parameter analysis on half plane and strip}, we discuss in detail the range of $\mu,\lambda$ for which $I_{\mu,\lambda}<0$ in the case $\O=\R^2_+$ and $\O=\S$ with suitable $p$. For $\Omega=\D$, we do not carry out a corresponding discussion, because in that approach only the negativity of $I_{\mu,\lambda}$ is needed and no additional properties are required.

\hspace*{\fill} 

For studies in the $L^2$ or on $\mathbb{R}^2_+$, homogeneity or spatial scaling methods always yield continuity of $I_{\mu,\lambda}$ with respect to $\mu$. This continuity is also indispensable for excluding the dichotomy case when proving the existence of a minimizer via Lions' compactness principle. According to Remark \ref{remark:uniformly bounded}, the minimiziation problem \eqref{def:minimizing problem} equipped with a $L^p$ upper bound in terms of $\mu$ and $\lambda$ in $K_{\mu}$, thereby we are able to prove continuity of  $I_{\mu,\lambda}$ with respect to $\mu$ in a more general framework.

\begin{prop}\label{prop:continuous of I_mu}
    For any $p>4/3$ and $(\mu,\lambda)$ satisfy \eqref{ineq:negativity of I_mu}, the map $\mu \mapsto I_{\mu,\lambda}$
    is continuous.
\end{prop}
\begin{proof}
    Recall the observation \eqref{ineq:minimizing seq uniformly bounded}. One can rewrite the minimizing problem \eqref{def:minimizing problem} for any $(\mu,\lambda)$ satisfy the \eqref{ineq:negativity of I_mu} into:
    \begin{align}
        I_{\mu,\lambda}=\inf\left\{ E_{p,\lambda}[\omega] \mid \omega\in K_\mu \tt{ and } \|\omega\|_p\leq C\lambda^{\frac{3p-2}{3p^2-4p}}\mu^{\frac{4p-4}{3p^2-4p}} \right\}. \label{def:alternative minimizing problem}
    \end{align}
    Doing so, for any $(\tau,\lambda)$ satisfy \eqref{ineq:negativity of I_mu}, 
    \begin{align*}
        I_{\mu,\lambda} & =\inf\left\{ -E_{p,\lambda}[\omega] \mid \omega\in K_\mu \tt{ and } \|\omega\|_p\leq C\lambda^{\frac{3p-2}{3p^2-4p}}\mu^{\frac{4p-4}{3p^2-4p}} \right\}, \\ 
        & =\inf\left\{ -E_{p,\lambda}[\frac{\mu}{\tau}\omega] \mid \omega\in K_\tau \tt{ and } \|\omega\|_p\leq C\lambda^{\frac{3p-2}{3p^2-4p}}\tau \mu^{\frac{4p-4}{3p^2-4p}-1} \right\}, \\
        & =\frac{\mu^2}{\tau^2} \inf\left\{  -E_{p,\lambda}[\omega] - \frac{1}{p\lambda}\left( 1 - \frac{\mu^{p-2}}{\tau^{p-2}} \right)\|\omega\|_p^p \mid \omega\in K_\tau \tt{ and } \|\omega\|_p\leq C\lambda^{\frac{3p-2}{3p^2-4p}}\tau \mu^{\frac{4p-4}{3p^2-4p}-1} \right\}, \\ 
        & \geq \frac{\mu^2}{\tau^2} I_{\tau,\lambda} - C\frac{1}{p \lambda}\left| 1 - \frac{\mu^{p-2}}{\tau^{p-2}} \right|\lambda^{\frac{3p-2}{3p^2-4p}}\tau \mu^{\frac{4p-4}{3p^2-4p}-1}.
    \end{align*}
    Thus 
    \begin{align*}
        \lim_{\tau\rightarrow\mu}\left( I_{\mu,\lambda} - I_{\tau,\lambda}\right) 
        & \geq  \lim_{\tau\rightarrow\mu}\left( \left(\frac{\mu^2}{\tau^2} -1\right)I_{\tau,\lambda} - C\frac{1}{p \lambda}\left| 1 - \frac{\mu^{p-2}}{\tau^{p-2}} \right|\lambda^{\frac{3p-2}{3p^2-4p}}\tau \mu^{\frac{4p-4}{3p^2-4p}-1} \right), \\
        & = \lim_{\tau\rightarrow\mu}\left( \left(\frac{\mu^2}{\tau^2} -1\right)I_{\tau,\lambda} \right), \\
        & \geq -C\lim_{\tau\rightarrow\mu}\left( \left|\frac{\mu^2}{\tau^2} -1\right|\lambda^{\frac{2}{3p-4}}\tau^{\frac{4p-4}{3p-4}}\right), \\
        & =0,
    \end{align*}
    where we used the lowerbound for $I_{\tau,\lambda}$ given in \eqref{ineq:lowerbound of I_mu,lambda}. On the other hand, by a similar argument, we have 
    \begin{align*}
        \lim_{\tau\rightarrow\mu}\left( I_{\tau,\lambda} - I_{\mu,\lambda}\right) 
        & \geq  \lim_{\tau\rightarrow\mu}\left( \left(\frac{\tau^2}{\mu^2} -1\right)I_{\mu,\lambda} - C\frac{1}{p \lambda}\left| 1 - \frac{\tau^{p-2}}{\mu^{p-2}} \right|\lambda^{\frac{3p-2}{3p^2-4p}}\mu \tau^{\frac{4p-4}{3p^2-4p}-1} \right), \\
        & = \lim_{\tau\rightarrow\mu}\left( \left(\frac{\tau^2}{\mu^2} -1\right)I_{\mu,\lambda} \right), \\
        & =0.
    \end{align*}
    Hence, 
    \begin{align*}
        \lim_{\tau\rightarrow\mu}\left( I_{\tau,\lambda} - I_{\mu,\lambda}\right) =0,
    \end{align*}
    and the continuity is proved.
    
\end{proof}

\begin{prop}[General structure of minimizers]
\label{prop:general structure}
Let $p>1$, $(\mu,\lambda)$ satisfy \eqref{ineq:negativity of I_mu}. Each minimizer $\omega \in S_{\mu,\lambda}$ satisfies
\begin{equation}
\label{eqn:general structure of minimizer}
\omega^{p-1} = \lambda(\psi -W x_{2})_+, \qquad 
\psi (x) = \int_{\Omega}G(x,y)\omega (y)\mathrm{d}y,
\end{equation}
for some constants $W \in \R$, uniquely determined by $\omega$.
\end{prop}
\noindent\textit{Proof: } The proof follows from a standard argument, e.g., \cite{friedman1982variational, friedman1981vortex} for vortex rings, and is a direct consequence of Proposition 2.5 in \cite{abe2022stability} or Proposition 2.6 in \cite{Dong2026StabilityOV} but without the limitation on $L^1$-norm.

Take an arbitrary minimizer $\omega \in S_{\mu,\lambda}$. Since $I_{\mu,\lambda}< 0$ by \eqref{ineq:negativity of I_mu}, the minimizer $\omega\not\equiv 0$. There exists a $\delta_{0} > 0$ such that $Vol(\{x\in \R^2_+\mid \omega \geq \delta_{0}\}) > 0$. Fix a compactly supported $h_{1}\in L^{\infty}(\Omega)$ such that $\spt h_{1}\subset \{\omega \geq \delta_{0}\}$ and
\begin{align*}
    \int_{\Omega}h_{1}(x)\mathrm{d}x = 0, && \int_{\Omega}x_{2}h_{1}(x)\mathrm{d}x = 1.
\end{align*}

For any $\delta \in (0, \delta_{0})$ and compactly supported $h \in L^{\infty}(\Omega)$ such that $h \geq 0$ on $\{0 \leq \omega \leq \delta \}$, define
\begin{align*}
    \eta = h - \left(\int_{\Omega}x_2h\mathrm{d}x\right)h_{1},
\end{align*}
so that $\int x_{2}\eta \mathrm{d} x = 0$. For sufficiently small $\eps>0$, $\omega+\epsilon\eta\geq \delta-\epsilon\|\eta\|_{\infty}\geq 0$ on $\{\omega > \delta \}$. Since $\eta = h\geq 0$ on $\{0\leq \omega\leq \delta\}$, we also have $\omega+\epsilon\eta\geq 0 $ on $\{0\leq \omega\leq \delta\}$. Hence $\omega+\epsilon\eta\in K_\mu$ for small $\eps>0$. Since $\omega$ is a minimizer of \eqref{def:minimizing problem}, then
\begin{align*}
    0\geq & \frac{d}{d\eps}\Bigg|_{\eps=0}E_{p,\lambda}[\omega+\eps\eta] = \int_{\Omega}\left(\psi-\frac{1}{\lambda}\omega^{p-1}\right)\eta = \int_{\Omega}\left(\psi-\frac{1}{\lambda}\omega^{p-1}\right)\left(h - \left(\int_{\Omega}x_2h\mathrm{d}x\right)h_{1}\right), \\
    = & \int_{\Omega}\left(\psi-Wx_2-\frac{1}{\lambda}\omega^{p-1}\right)h = \int_{\omega>\delta}+\int_{0< \omega \leq \delta}\left( \Psi-\frac{1}{\lambda}\omega^{p-1} \right)h,
\end{align*}
where we use $\Psi=\psi-W x_2$ and
\begin{align*}
    W=\int_{\Omega}\left(\psi-\frac{1}{\lambda}\omega^{p-1}\right)h_1 dx.
\end{align*}

Since $h \in L^{\infty}(\Omega)$ is an arbitrary function satisfying $h \geq 0$ on $\{0 \leq \omega \leq \delta\}$, then 
\begin{align}
    \left\{\begin{aligned}
        \Psi - \frac{1}{\lambda}\omega^{p-1} & = 0 \quad \text{on } \{\omega > \delta\}, \\
        \Psi -  \frac{1}{\lambda}\omega^{p-1} & \leq 0 \quad \text{on } \{ 0 \leq \omega \leq  \delta\}. 
    \end{aligned}\right.
    \label{eqn:Psi and omega with para delta without L1}
\end{align}

Since $\delta>0$ is arbitrary, sending $\delta \rightarrow 0$ implies
\begin{align}
    \left\{\begin{aligned}
    \Psi - \frac{1}{\lambda}\omega^{p-1} & = 0 \quad \text{on } \{\omega > 0\}, \\
    \Psi & \leq 0 \quad \text{on } \{ \omega =0\}.
    \end{aligned}\right.
    \label{eqn:Psi and omega passing delta to 0 without L1}
\end{align}

Thus $\omega^{p-1} = \lambda(\psi -W x_{2})_+$. To show the uniqueness, suppose there exists some $W_*,\gamma_*$, such that $\omega = \lambda(\psi - W_* x_{2})_+$, then $\Psi_*=\psi- W_* x_{2}$ satisfies \eqref{eqn:Psi and omega with para delta without L1} for $\delta\in (0,\delta_0).$ Hence 
\begin{align*}
    0\geq \int(\Psi_*-\frac{1}{\lambda}\omega) h dx = \int_{\Omega}\left(\psi-W_* x_2-\frac{1}{\lambda}\omega^{p-1}\right)h,
\end{align*}
for compactly supported $h \in L^{\infty}(\Omega)$ such that $h \geq 0$ on $\{0 \leq \omega \leq \delta \}$. By taking $h=\pm h_1$, we have 
\begin{align*}
    W_*=\int_{\Omega}\left(\psi-\frac{1}{\lambda}\omega^{p-1}\right)h_1=W.
\end{align*}
The uniqueness is obtained. 

\rightline{$\Box$}

\section{Studies on $\mathbb{R}^2_+$ and $\S$} \label{sec:study of half plane and strip}

This section focuses on the General Convergence Theorem on $\R^2_+$ and $\S$. To achieve this goal, we first discuss the ranges of $\mu$ and $\lambda$ indicated by \eqref{ineq:negativity of I_mu} such that the lower bound of $\lambda$ is independent of the choice of $\mu$. This ensures the validity of Proposition \ref{prop:strictly subadditivity} and consequently yields the first two cases of Theorem \ref{thm:general convergence theorem without L1}. Moreover, based on the established ranges of $p,\mu,$ and $\lambda$, we present further properties of minimizers, such as Proposition \ref{prop:regularity study} and \ref{prop:compactly support of minimizer}. It is worth noting that the parameter range identified in Section \ref{sec:parameter analysis on half plane and strip} is same as the range stated in Case \ref{case:general convergence on half-plane} and \ref{case:general convergence on strips} of Theorem \ref{thm:general convergence theorem without L1}.

\subsection{Discussion about parameters $\mu$ and $\lambda$} \label{sec:parameter analysis on half plane and strip}

\begin{cor}\label{cor:range of mu and lambda Omega=R^_+}
    Suppose $p>4/3$ that and $\O=\R^2_+$, the range of parameters $(\mu,\lambda)$ given by \eqref{ineq:negativity of I_mu} is $\R_+ \times \R_+$.
\end{cor}
\begin{proof}
    For any $R>0$, choose $K=B(0,R)\cap \R^2_+$ and define 
\begin{align*}
    J_{\R^2_+}=\int_{K}\int_{K}G_H(x,y)dxdy=\frac{1}{4\pi}\int_{K}\int_{K}\ln(1+\frac{4x_2 y_2}{|x-y|^2})dxdy.
\end{align*}
Then 
\begin{align}
    -2\pi J_{\R^2_+} = \frac{1}{2}\int_{\R^2_+}\int_{\R^2_+} v(x)v(y)\ln|x-y|dxdy, \label{eqn:write integral on K in terms of v}
\end{align}
where $v(x)=\sgn(x_2)\cdot \mathbf{1}_{B(0,R)}(x)$. In $B(0,R)$, the function $v$ can be written in a radial form,
\begin{align}
    v(\theta)=\sum_{n \tt{ is odd}}\frac{4}{n\pi}\sin(n\theta). \label{eqn:radial form of v}
\end{align}
Write $x=x(r,\theta)$ and $y=y(\rho,\varphi)$, and in almost everywhere sense, it is sufficient to consider $r\neq \rho$ inside the integral \eqref{eqn:write integral on K in terms of v}. Set 
\begin{align*}
    \alpha=\frac{\min\{r,\rho\}}{\max\{r,\rho\}}\in[0,1),
\end{align*}
then the Fourier series of $\ln|x-y|$ is
\begin{align*}
    \ln|x-y|=\ln(\max\{r,\rho\})-\sum_{m=1}^\infty \frac{\alpha^m}{m}\cos(m(\theta-\varphi)).
\end{align*}

We can now evaluate \eqref{eqn:write integral on K in terms of v}:
\begin{align*}
    -4\pi J_{\R^2_+} 
    & = \iint\limits_{[0,R]^2}\iint\limits_{[0,2\pi]^2}v(\theta)v(\varphi)\left( \ln(\max\{r,\rho\})-\sum_{m=1}^\infty \frac{\alpha^m}{m}\cos(m(\theta-\varphi)) \right)r\rho ~ d\theta d\varphi  dr d\rho.
\end{align*}
Since $v$ is a sine series independent of $r,\rho$, we have for any $r,\rho$, 
$$\int_{0}^{2\pi}\int_{0}^{2\pi} v(r,\theta)v(\rho,\varphi)\ln(\max\{r,\rho\}) d\theta d\varphi=0.$$
To exchange summation and integration, it is enough to compute each term one by one and then check the finiteness of the sum. Specifically,
\begin{align*}
    \sum_{m=1}^\infty &  \int_{0}^{R}\int_{0}^{R}\int_{0}^{2\pi}\int_{0}^{2\pi} v(\theta)v(\varphi) \frac{\alpha^m}{m}\cos(m(\theta-\varphi)) r\rho d\theta d\varphi dr d\rho , \\
    & = \sum_{m=1}^\infty \frac{1}{m} \left( \int_{0}^{R}\int_{0}^{R} r\rho \alpha^m dr d\rho \right) \left( \int_{0}^{2\pi}\int_{0}^{2\pi} v(\theta)v(\varphi)\cos(m(\theta-\varphi)) d\theta d\varphi \right) , \\
    & = \sum_{m=1}^\infty \frac{1}{m} \left( \frac{R^4}{2(m+2)} \right) \left( \left( \int_{0}^{2\pi}v(\theta)\cos(m \theta) d\theta \right)^2 - \left( \int_{0}^{2\pi}v(\theta)\sin(m \theta) d\theta \right)^2 \right).
\end{align*}
Recall the radial form of $v(\theta)$ and identities for $m,n\in \mathbb{Z}_{\geq 1}$,
\begin{align*}
    & \int_0^{2\pi}\cos(m\theta)\sin(n\theta)d\theta=0, 
    && \int_0^{2\pi}\sin(m\theta)\sin(n\theta)d\theta= 
    \left\{ \begin{aligned}
        & 0, && m\neq n, \\
        & \pi, && m=n.
    \end{aligned} \right.
\end{align*}
Thus 
\begin{align*}
    \sum_{m \tt{ is odd}}^\infty &  \int_{0}^{R}\int_{0}^{R}\int_{0}^{2\pi}\int_{0}^{2\pi} v(\theta)v(\varphi) \frac{\alpha^m}{m}\cos(m(\theta-\varphi))  r\rho \, d\theta d\varphi dr d\rho , \\
    & = -\sum_{m \tt{ is odd}} \frac{1}{m} \frac{R^4}{2(m+2)} \frac{16}{m^2} = -8 R^4 \sum_{m \tt{ is odd}} \frac{1}{m^3(m+2)} < \infty.
\end{align*}

In conclusion, 
\begin{align*}
    J_{\R^2_+} 
    & = \frac{2 R^4 }{\pi}\sum_{m \tt{ is odd}} \frac{1}{m^3(m+2)}.
\end{align*}
Consequently, the constant $\lambda_{K,\mu}$ can be written as
\begin{align}
    \lambda_{K,\mu}=\frac{\mu^{p-2}}{2\cdot 3^{2-p}p \sum_{m \tt{ is odd}}^\infty \frac{1}{m^3(m+2)}}R^{4-3p}. \label{eqn:reduced lambda_{mu,K} for half plane}
\end{align}
Letting $R\to \infty$, we derived the range of $(\mu,\lambda)$ described by \eqref{ineq:negativity of I_mu} is $\R_+\times \R_+$.
    
\end{proof}

Unlike the half‑plane, the study on the strip is constrained by its geometric features, and a lower bound for $\lambda$ independent of $\mu$ can only be obtained when $p\ge 2$. When $p>2$ this lower bound is $0$, and when $p=2$ it is a constant depending on the strip width $L$.

\begin{cor}\label{cor:range of mu and lambda Omega=S}
    Suppose that $p > 2$ and $\O=\S$, the range of $(\mu,\lambda)$ given by \eqref{ineq:negativity of I_mu} is $\R_+ \times \R_+$. When $p=2$, this range contains $\R_+ \times (\frac{2\pi^4}{L^2},\infty)$.
\end{cor}
\begin{proof}
    For any $R>L/\pi$, choose $K=(-R, R) \times (0, L)$. Then 
    \begin{align*}
        \int_{K}\int_{K} G_S(x,y)dxdy & \geq \frac{1}{4\pi }\int_{K}\int_{K}\ln\left(1+\sin(\frac{\pi}{L}x_2)\sin(\frac{\pi}{L}y_2)\exp(-\frac{\pi}{L}|x_1-y_1|)\right)dxdy, \\
        & \geq \frac{1}{8\pi}\left(\int_{0}^{L}\sin(\frac{\pi}{L}x_2)dx_2\right)^2\int_{-R}^{R}\int_{-R}^{R}\exp(-\frac{\pi}{L}|x_1-y_1|)dx_1dy_1, \\
        & = \frac{1}{8\pi }\left(\frac{2L}{\pi}\right)^2 \left( \frac{2\,L^2 \,{\left(\exp({-\frac{2R\pi}{L}}) -1\right)}}{\pi^2 }+\frac{4\,L\,R}{\pi } \right), \\
        & \geq \frac{1}{8\pi }\left(\frac{2L}{\pi}\right)^2 \left( -\frac{2 L^2}{\pi^2 }+\frac{4\,L\,R}{\pi } \right), \\
        & = \frac{L^4}{\pi^4} \left( -\frac{1}{\pi} + \frac{2R}{L} \right) \geq \frac{L^3 R}{\pi^4}.
    \end{align*}
    Thus the constant $\lambda_{K,\mu}$ admits the upper bound 
    \begin{align*}
        \lambda_{K,\mu}\leq \frac{4\mu^{p-2}LR}{p}\frac{\pi^4}{L^3 R}\left(R L^2\right)^{2-p} = \frac{4\pi^4}{p}\mu^{p-2}L^{2-2p}R^{2-p}.
    \end{align*}
    Plug in $p=2$, we prove the particular case. As for $p>2$, letting $R\to \infty$, shows that the range of $(\mu,\lambda)$ in \eqref{ineq:negativity of I_mu} is $\R_+\times \R_+$.
    
\end{proof}

\subsection{Further structural analysis} \label{sec:further structure analysis on half plane and strip}

We now investigate further properties of the minimizers $\omega\in S_{\mu,\lambda}$ under the parameter regimes for $p,\mu,\lambda$ discussed in Section \ref{sec:parameter analysis on half plane and strip}. Proposition \ref{prop:positive W} establishes the strict positivity of the constant $W$ in the structural analysis \eqref{eqn:general structure of minimizer}, while Proposition \ref{prop:compactly support of minimizer} shows that every minimizer has compact support. In the proof of Proposition \ref{prop:positive W}, we first show $W\ge 0$ via decay estimates, and then prove $W>0$ through separate arguments for $\mathbb{R}^2_+$ and strips.

\begin{prop}\label{prop:positive W}
    Suppose either \ref{case:general convergence on half-plane} or \ref{case:general convergence on strips} holds.
    For any $\omega\in S_{\mu,\lambda}$, the constant $W=W(\omega)$ given by the structural analysis \eqref{eqn:general structure of minimizer} is strictly positive.
\end{prop}
\begin{proof}
    First we prove that $W\ge0$. For any $\omega\in S_{\mu,\lambda}$, take a sequence $\{\omega_n\}\subseteq C_c^\infty(\O)$ such that $x_2\omega_n\to x_2\omega$ in $L^1(\O)$ and $\omega_n\to \omega$ in $L^p(\O)$, then we have 
    \begin{align*}
        0\leq \psi(x)= \int_{\O}G(x,y)\left( \omega(y)-\omega_n(y) \right)dy + \int_{\O}G(x,y)\omega_n(y) dy,
    \end{align*}
    by \eqref{ineq:sup estimate of stream fct} and the comparison \eqref{ineq:comparison between green's function},
    \begin{align*}
        0\leq \psi(x)\leq C \|x_2(\omega - \omega_n)\|_1^{\frac{2p-2}{3p-2}}\|\omega-\omega_n\|_p^{\frac{p}{3p-2}} + \frac{x_2\|y_2\omega_n\|_1}{\pi \inf_{y\in \spt\omega_n}|x-y|^2} ,
    \end{align*}
    Sending $|x|\to \infty$ and then $n\to \infty$ implies $\psi\to 0 $ as $|x| \to \infty$. 
    
    Choose a sequence $y_n=(n,L)\in \O $, by \eqref{eqn:Psi and omega passing delta to 0 without L1} and $\psi, \omega \to 0$ as $|x|\to \infty$,
    \begin{align*}
        \limsup_{n\to \infty} \left( \psi(n,L)-WL - \frac{1}{\lambda}\omega^{p-1}(n,L) \right) \leq 0.
    \end{align*}
    which force $WL \geq 0$, and then $W\geq 0$.

    \underline{Case $\Omega=\R^2_+$.} Recall the following identity for arbitrary functions
    \begin{align*}
        \nabla \cdot \left( \nabla\psi(x\cdot\nabla\psi) - \frac{1}{2}x\left|\nabla\psi\right|^2\right)=\Delta\psi(x\cdot\nabla\psi).
    \end{align*}
    For a maximizer $\omega\in S_{\mu,\lambda}$ to \eqref{def:minimizing problem} and $\psi$ defined in \eqref{def:stream function}, the function $\Psi=\psi-Wx_2$ satisfies 
    $$-\Delta \Psi = \omega=\lambda^{\frac{1}{p-1}}\Psi_+^{\frac{1}{p-1}}.$$
    For convenience, let $\gamma=\frac{1}{p-1}$. Then
    \begin{align*}
        \Delta\psi\,(x\cdot\nabla\psi) &= -\lambda^{\gamma} \Psi_+^{\gamma} (x\cdot\nabla\Psi) - x_2 W \omega \\
        &= -\frac{\lambda^\gamma}{\gamma+1}\, x\cdot \nabla (\Psi_+^{\gamma+1}) - x_2 W \omega \\
        &= -\frac{\lambda^\gamma}{\gamma+1}\,\nabla\cdot(x \Psi_+^{\gamma+1}) + \frac{2\lambda^\gamma}{\gamma+1} \Psi_+^{\gamma+1} - x_2 W \omega,\\
        &= -\frac{1}{\gamma+1} \nabla\cdot (x\,\Psi\omega) + \frac{2}{\gamma+1}\psi\omega - \frac{\gamma+3}{\gamma+1} W x_2\omega.
    \end{align*}
    Thus
    \begin{align*}
        \nabla \cdot \Big( \nabla\psi\,(x\cdot\nabla\psi) - \frac12 x |\nabla\psi|^2 + \frac{1}{\gamma+1} (x\,\Psi\omega) \Big) = \frac{2}{\gamma+1}\psi\omega - \frac{\gamma+3}{\gamma+1} W x_2\omega.
    \end{align*}
    The normal trace of the vector field on the right-hand side on $\p \R^2_+$ vanishes by $\psi\equiv 0 $. By using $\nabla\psi\in L^2(\R^2_+)$ and $\psi\omega\in L^1(\R^2_+)$, the integral of the left-hand side vanishes. Thus,
    \begin{align*}
        E[\omega]=\frac{\gamma+3}{4}W\mu.
    \end{align*}
    Meanwhile, multiplying \eqref{eqn:general structure of minimizer} by $\omega$ and integration on $\R^2_+$,
    \begin{align*}
        \|\omega\|_p^p=\lambda\left( \int_{\spt\omega}\psi\omega-W \int_{\spt\omega}x_2\omega \right)=\lambda\left( 2 E[\omega]-W\mu \right)=\frac{\gamma+1}{2}\lambda W\mu,
    \end{align*}
    then
    \begin{align*}
        I_{\mu,\lambda}=-E[\omega]+\frac{1}{p\lambda}\|\omega\|_p^p=\left(-\frac{\gamma+3}{4}+\frac{\gamma+1}{2p}\right)W\mu = - \frac{3p-4}{4(p-1)}W\mu,
    \end{align*}
    and
    \begin{align*}
        W=\frac{4(p-1)}{(4-3p)\mu} I_{\mu,\lambda} >0.
    \end{align*}

    \underline{Case $\Omega=\S$ and $p>2$.} We argue by contradiction. Suppose $W=0$, we have $\omega^{p-1}=\lambda \psi$. So that $\psi$ satisfies
    \begin{align*}
        \left\{\begin{aligned}
            -\Delta\psi & =\lambda^{\frac{1}{p-1}}\psi^{\frac{1}{p-1}} && \tt{ in } \S, \\
            \psi & =0 && \tt{ on } \p\S, \\
            \psi & \in L^{\frac{p}{p-1}}(\S).
        \end{aligned}\right.
    \end{align*}
    where the last constraint comes from $\omega\in L^p(\S)$. Multiplying the equation by $\psi$ and integrating by parts yields,
    \begin{align}
        \int_{\S}(\p_1^2 \psi)^2+ (\p_2^2 \psi)^2 dx = \int_{\S}|\nabla \psi|^2 dx = \lambda^{\frac{1}{p-1}}\int_{\S}\psi^{\frac{p}{p-1}} dx. \label{eqn:identity 1 in proof W>0 on S}
    \end{align}

    Next, multiply the equation by $\partial_1\psi$ and integration in $x_2$,
    \begin{align*}
        - \int_{0}^{L} \p_1\psi ( \p_1^2 \psi + \p_2^2\psi ) dx_2 = \int_{0}^{L}  \psi^{\frac{1}{p-1}} \p_1 \psi dx_2.
    \end{align*}
    Integration by part in $x_2$, the identity can be written as:
    \begin{align*}
        \frac{1}{2}\frac{d}{dx_1} \int_{0}^{L} - (\p_1\psi)^2 + (\p_2\psi)^2 dx_2 & = \int_{0}^{L} -\p_1\psi \p_1^2 \psi + \p_2\psi \p_1\p_2 \psi dx_2, \\ 
        & =\left[ \p_1 \psi \p_2 \psi \right]_{x_2=0}^{x_2=L} -\int_{0}^{L} \p_1\psi \p_1^2 \psi + \p_1\psi \p^2_2 \psi dx_2, \\
        & =0+ \lambda^{\frac{1}{p-1}}\int_{0}^{L}  \psi^{\frac{1}{p-1}} \p_1 \psi dx_2, \\
        & =\frac{\lambda^{\frac{1}{p-1}}(p-1)}{p}\frac{d}{dx_1} \int_{0}^{L}  \psi^{\frac{p}{p-1}} dx_2.
    \end{align*}
    Therefore there exists a constant $F$ independent from $x_1$ such that
    \begin{align*}
        \int_{0}^{L} - (\p_1\psi)^2 + (\p_2\psi)^2 -\frac{2\lambda^{\frac{1}{p-1}}(p-1)}{p}\psi^{\frac{p}{p-1}} dx_2 = F.
    \end{align*}
    Integrating over $\mathbb{R}$ gives
    \begin{align*}
        \int_{\S} - (\p_1\psi)^2 dx + \int_{\S} (\p_2\psi)^2  dx-\frac{2\lambda^{\frac{1}{p-1}}(p-1)}{p}\int_{\S} \psi^{\frac{p}{p-1}} dx = \int_{\R}F dx_1.
    \end{align*}
    The left hand side is finite by \eqref{eqn:identity 1 in proof W>0 on S} and $\psi \in L^{\frac{p}{p-1}}(\S)$. So that $\int_{\R}F dx_1<+\infty$, hence $F=0$. We thus obtain
    \begin{align}
        \int_{\S} - (\p_1\psi)^2 dx + \int_{\S} (\p_2\psi)^2  dx = \frac{2\lambda^{\frac{1}{p-1}}(p-1)}{p}\int_{\S} \psi^{\frac{p}{p-1}} dx. \label{eqn:identity 2 in proof W>0 on S}
    \end{align}
    Observe that \eqref{eqn:identity 1 in proof W>0 on S} + \eqref{eqn:identity 2 in proof W>0 on S} implies
    \begin{align*}
        \int_{\S}(\p_2\psi)^2 dx = \lambda^{\frac{1}{p-1}}\left( \frac{p-1}{p} +\frac{1}{2} \right) \int_{\S} \psi^{\frac{p}{p-1}} dx.
    \end{align*}
    Recall the assumption $p>2$ and identity \eqref{eqn:identity 1 in proof W>0 on S}
    \begin{align*}
        0 \leq \int_{\S}(\p_1\psi)^2 dx = \lambda^{\frac{1}{p-1}}\left( 1 - \frac{p-1}{p} - \frac{1}{2} \right) \int_{\S} \psi^{\frac{p}{p-1}} dx = \lambda^{\frac{1}{p-1}}\left( \frac{1}{p}-\frac{1}{2} \right) \int_{\S} \psi^{\frac{p}{p-1}} dx \leq 0,
    \end{align*}
    which implies  
    \begin{align*}
        \int_{\S} \psi^{\frac{p}{p-1}} dx =0.
    \end{align*}
    Hence $\psi\equiv0$ and $\omega=0$, contradicting $I_{\mu,\lambda}<0$.
    
    \underline{Case $\Omega=\S$ and $p=2$.} The proof is identical to that of Corollary 4.3 in \cite{Dong2026StabilityOV} and is omitted here.
    
\end{proof}


\begin{remark}
    Since every minimizer $\omega \in S_{\mu,\lambda}$ satisfies \eqref{eqn:general structure of minimizer} with $W>0$, it obeys the same elliptic equation as the kinetic energy maximizers studied in \cite{abe2025existence}. Consequently, all structural results obtained in \cite{abe2025existence} for those maximizers obtained through the structure of the equation can be carry over directly to the minimizers of the penalized energy.
\end{remark}


For $\alpha \in (0, 1)$ and a positive integer $m$, let $C^{m,\alpha}(\O)$ be the space of functions $f$ such that $\partial^l_x f$ is uniformly bounded and continuous for $|l| \le m-1$, and $\partial^l_x f$ is $\alpha$-Hölder continuous for $|l| = m$. We denote by $\lceil s \rceil$ the smallest integer greater than or equal to $s \in \mathbb{R}$.

\begin{prop}\label{prop:regularity study}
    Suppose either \ref{case:general convergence on half-plane} or \ref{case:general convergence on strips} holds.
    Then, for any $\omega \in S_{\mu,\lambda}$, $\omega \in L^\infty(\O)\cap C^{m,\alpha}(\O)$ and $\psi = \int_{\O}G\omega \in L^\infty(\O) \cap W^{2,q}_{\mathrm{loc}}(\O)$ for all $q \in (1, \infty)$, where 
    \begin{align}
        m = \left\lceil \frac{1}{p-1} - 1 \right\rceil, \quad \alpha = \frac{1}{p-1} - m \in (0, 1]. \label{eqn:regularity of minimizers}
    \end{align}
    In particular, $\psi \in C^{1,\beta}(\O)$ for all $\beta \in (0, 1)$.
\end{prop}
\begin{proof}
    The proof is same as the argument of Lemma 4.1 in \cite{abe2025existence}, since the structure of functions study there share a same structure of minimizers to the penalized energy given in \eqref{eqn:general structure of minimizer}.
    
\end{proof}

\begin{prop}\label{prop:compactly support of minimizer}
    Suppose either \ref{case:general convergence on half-plane} or \ref{case:general convergence on strips} holds. Then, for any $\omega \in S_{\mu,\lambda}$,  $\spt \omega = \overline{\{x \in \mathbb{R}^2_+ : \omega > 0\}}$ is compact.
\end{prop}
\begin{proof}
    By \eqref{eqn:general structure of minimizer}, $\{x \in \mathbb{R}^2_+ : \omega > 0\} = \{x \in \mathbb{R}^2_+ : \psi(x)/x_2 > W \}$. By Proposition \ref{prop:regularity study}, the function 
    \begin{align*}
        \frac{\psi(x_1, x_2)}{x_2} = \int_0^1 \partial_{x_2} \psi(x_1, x_2 s)\,ds
    \end{align*}
    is bounded and uniformly continuous in $\O$. Since $\psi/x_2 \in L^2(\O)$ by $\nabla \psi \in L^2(\O)$ and the Hardy's inequality, $\psi/x_2 \to 0$ as $|x| \to \infty$, the assertion follows.
    
\end{proof}

\subsection{Strictly Subadditivity of $I_{\mu,\lambda}$}

The goal of this subsection is to prove the strictly subadditivity of $I_{\mu,\lambda}$ in $\mu$, i.e. 

\begin{prop}\label{prop:strictly subadditivity}
    Suppose either \ref{case:general convergence on half-plane} or \ref{case:general convergence on strips} holds. Then for any $0<\alpha<\mu$,
    \begin{align}
        I_{\mu,\lambda}<I_{\mu-\alpha,\lambda}+I_{\alpha,\lambda}. \label{ineq:strictly subadditivity}
    \end{align}
\end{prop}

To achieve this goal, it is necessary to discuss separately according to cases \ref{case:general convergence on half-plane} and \ref{case:general convergence on strips}. Suppose \ref{case:general convergence on half-plane} holds, i.e., $\O=\R^2_+$, we use a scaling argument to reduce $I_{\mu,\lambda}$ to $I_{1,\lambda}$ and complete the proof based on its coefficients. Suppose \ref{case:general convergence on strips} holds, i.e., $\O=\S$, the situation is more involved. We first recall the Steiner symmetrization of $\omega$ and the non-decreasing of the penalized energy under such symmetrization. Then show that the kinetic energy of any $\omega$ invariant under Steiner symmetrization is concentrated in a bounded domain. Based on this, we prove the existence of a minimizer for \eqref{def:minimizing problem}, thereby completing the proof of strict subadditivity.

\subsubsection{Strictly Subadditivity for $\O=\R^2_+$}\hfill

\noindent\underline{\textit{Proof of the Proposition \ref{prop:strictly subadditivity} with \ref{case:general convergence on half-plane}:}}

Observe that for any $\tau>0$, $G(\tau x,\tau y)=G(x,y)$. By the scaling 
    \begin{align*}
        \overline\omega(x)=\mu^{-\frac{2}{3p-4}}\omega(\mu^{\frac{p-2}{3p-4}}x), \tt{ for any } \omega \in K_{\mu},
    \end{align*}
    the penalized energy transforms as
    \begin{align*}
        E_{p,\lambda}[\overline\omega]=\mu^{-\frac{4}{3p-4}}\mu^{-\frac{4p-8}{3p-4}}E[\omega]-\frac{1}{p\lambda}\mu^{-\frac{2p}{3p-4}}\mu^{-\frac{2p-4}{3p-4}}\|\omega\|_p^p=\mu^{-\frac{4p-4}{3p-4}}E_{p,\lambda}[\omega],
    \end{align*}
    the infimum satisfies
    \begin{align*}
        I_{\mu,\lambda}=\mu^{\frac{4p-4}{3p-4}}I_{1,\lambda},
    \end{align*}
    which directly conclude \eqref{ineq:strictly subadditivity}.

\rightline{$\Box$}

\subsubsection{Strictly Subadditivity for $\O=\S$}

\begin{prop}[Steiner symmetrization]\label{prop:steiner symmetrization}
For $\omega \ge 0$ satisfying $\omega \in L^p(\S)$ and $x_2\omega \in L^1(\S)$, there exists $\omega^* \ge 0$ such that
\begin{align}
    \begin{aligned}
        &\omega^*(x_1,x_2) = \omega^*(-x_1,x_2), \\
        &\omega^*(x_1,x_2) \text{ is non-increasing for } x_1 > 0.
    \end{aligned}\label{def:symmetry property}
\end{align}
Moreover,
\begin{align}
    \begin{aligned}
        \|\omega^*\|_p &= \|\omega\|_p, \quad \forall p\geq 1, \\
        \|x_2\omega^*\|_1 &= \|x_2\omega\|_1, \\
        E[\omega^*] &\ge E[\omega].
    \end{aligned}
\end{align}
\end{prop}
\begin{proof}
    Recall the structure of the Green's function for $\S$ \eqref{def:green's function for strip}. For any fixed $x_2,y_2$, $G$ is a strictly decreasing function with respect to $|x_1-y_1|$. Then the Riesz Rearrangement inequality in the $x_1$-direction can be applied, which directly implies this theorem. See \cite{Fraenkel1974}, Appendix I, and \cite{TurkingtonBruce1983Osvf}, p.1053.

\end{proof}

\begin{prop}\label{prop:decay in x1 estimate}
    Let $p>1$, $\omega\in L^p(\S)$ and $x_2 \omega \in L^1(\S)$. Suppose \eqref{def:symmetry property} holds for $\omega$, then for any $A\geq 2$, $x_2\leq \frac{|x_1|}{A}$, we have 
    \begin{align}
        \psi(x)\leq C\left( x_2^{2(1-\frac{1}{p})} A^{-\frac{1}{p}}\|\omega\|_p + x_2 \left( \frac{A}{|x_1|} \right)^2 \|x_2\omega\|_1 \right).\label{ineq:decay in x1 estimate}
    \end{align}
    The constant $C$ is independent of $\omega$ and $A$.
\end{prop}
\begin{proof}
We may assume that $x_1>0$, and 
\begin{align*}
    \psi(x) & =\int_{|x_1-y_1|\leq \frac{x_1}{A}} + \int_{|x_1-y_1|\geq \frac{x_1}{A}} G(x,y)\omega(y)dy \leq \int_{|x_1-y_1|\leq \frac{x_1}{A}} + \int_{|x_1-y_1|\geq \frac{x_1}{A}} G_H(x,y)\omega(y)dy, \\
    & \leq \left( \int_{\R^2_+} G_H(x,y)^{p'} dy \right)^\frac{1}{p'} \left( \int_{|x_1-y_1|\leq \frac{x_1}{A}} \omega(y)^{p} dy \right)^\frac{1}{p} + \frac{x_2}{\pi}\left( \frac{A}{|x_1|} \right)^2 \int_{|x_1-y_1|\geq \frac{x_1}{A}} y_2\omega(y)dy, \\
    & \leq Cx_2^{2(1-\frac{1}{p})} \left( \int_{|x_1-y_1|\leq \frac{x_1}{A}} \omega^{p}(y)dy \right)^\frac{1}{p} + \frac{x_2}{\pi}\left( \frac{A}{|x_1|} \right)^2 \|x_2\omega\|_1,
\end{align*}
where $\frac{1}{p}+\frac{1}{p'}=1$. Since \eqref{def:symmetry property} holds for $\omega$, then
\begin{align*}
    \omega(y)^p \leq \frac{1}{y_1}\int_0^{y_1} \omega(t,y_2)^p dt \leq \frac{1}{y_1}\int_0^{\infty} \omega(t,y_2)^p dt,
\end{align*}
so that
\begin{align*}
    \int_{|x_1-y_1| \leq \frac{x_1}{A}} \omega^{p}(y)dy 
    \leq & \int_{0}^L \int_{x_1(1-\frac{1}{A})}^{x_1(1+\frac{1}{A})} \frac{1}{y_1} \int_0^{\infty} \omega(t,y_2)^p dt dy_1 dy_2,\\
    \leq & \frac{4}{A} \int_{0}^L \int_0^{\infty} \omega(t,y_2)^p dt dy_2 = \frac{2}{A} \|\omega\|_p^p.
\end{align*}

Hence we have obtained \eqref{ineq:decay in x1 estimate}.
    
\end{proof}

The decay in $x_1$ estimate \eqref{ineq:decay in x1 estimate} will now be shown to imply that the kinetic energy $E[\omega]$ is concentrated on a bounded domain $Q=\{x \in \S \mid |x_1|\geq AR, x_2<R \}$.
\begin{prop}\label{prop:strip concentration of kinetic energy}
    Let $p\geq 2$, $\omega$ satisfy the assumption of Proposition \ref{prop:decay in x1 estimate},
    \begin{align}
        \int\limits_{S\setminus Q} \psi(x)\omega(x) dx\leq  C\left(R^{-\frac{p}{2p-1}}\|x_2\omega\|_1^{\frac{3p-2}{2p-1}}\|\omega\|_p^{\frac{p}{2p-1}} + A^{-\frac{1}{p}} R^{1-\frac{2}{p}} \|x_2\omega\|_1 \|\omega\|_p + R^{-2} \|x_2\omega\|_1^2 \right). \label{ineq:strip concentration of kinetic energy}
    \end{align}
    The constant $C$ is independent of $\omega$, $A$ and $R$.
\end{prop}
\begin{proof}
We decompose 
\begin{align*}
    \int_{S\setminus Q} \psi(x)\omega(x) dx = \int\limits_{x_2 \geq R} \psi(x)\omega(x) dx + \int\limits_{x_2 < R, \ |x_1|\geq AR} \psi(x)\omega(x) dx,
\end{align*}
and estimate by \eqref{ineq:estimate of stream fct},
\begin{align*}
    \int_{x_2 \geq R} \psi(x)\omega(x) dx \leq C \|x_2\omega\|_1^{\frac{p}{2p-1}}\|\omega\|_p^{\frac{p}{2p-1}} \int\limits_{x_2\geq R} x_2^{\frac{p-1}{2p-1}}\omega(x) dx \leq R^{-\frac{p}{2p-1}}\|x_2\omega\|_1^{\frac{3p-2}{2p-1}}\|\omega\|_p^{\frac{p}{2p-1}}.
\end{align*}

Since $|x_1|\geq AR$ and $x_2 < R$ imply $x_2 < \frac{|x_1|}{A}$, applying \eqref{ineq:decay in x1 estimate} yields
\begin{align*}
    \int\limits_{x_2 < R, \ |x_1|\geq AR} &  \psi(x)\omega(x) dx \\
    \leq & C\left( A^{-\frac{1}{p}}\|\omega\|_p \int\limits_{x_2 < R} x_2^{2(1-\frac{1}{p})}\omega(x) dx  +  \left( \frac{A}{|x_1|} \right)^2 \|x_2\omega\|_1 \int\limits_{\S} x_2\omega(x)dx \right), \\
    \leq & C\left( A^{-\frac{1}{p}} R^{1-\frac{2}{p}} \|x_2\omega\|_1 \|\omega\|_p + R^{-2} \|x_2\omega\|_1^2 \right),
\end{align*}
where the last inequality comes from the assumption $p\geq 2$, so that $2(1-\frac{1}{p})\geq 1$. \eqref{ineq:strip concentration of kinetic energy} follows.
    
\end{proof}

Proposition \ref{prop:strip concentration of kinetic energy} implies that the kinetic energy $E[\omega]$ is continuous by the weak continuity in a certain proper subset of $L^p$, as we now show.

\begin{prop}[Convergence of the kinetic energy]
\label{prop:strip energy convergence}
Assume that the sequence \(\{\omega_n\} \subset L^p(\S)\) satisfies the following:
\begin{itemize}
    \item \(\sup_n \|x_2\omega_n\|_1 + \sup_n \|\omega_n\|_p \leq M\) for some constants \(M\), and
    \item \(\omega_n \rightharpoonup \omega\) in \(L^p(\S)\).
\end{itemize}
If each $\omega_n$ is nonnegative and satisfies \eqref{def:symmetry property}, then
\(E[\omega_n] \to E[\omega] \) as \(n \rightarrow \infty\).
\end{prop}
\begin{proof}
    We decompose the energy into two terms
    \begin{align*}
        2E[\omega_n] = \int_{\S} \psi_n(x)\omega_n(x)\,dx = \int_{Q} \psi_n(x)\omega_n(x)\,dx + \int_{\S\setminus Q} \psi_n(x)\omega_n(x)\,dx,
    \end{align*}
    and observe that
    \begin{align*}
        \int_{Q} \psi_n(x)&\omega_n(x)dx \\
        &= \int_{Q} \omega_n(x) \int_{Q} G(x,y)\omega_n(y)\,dydx + \int_{Q} \omega_n(x)\int_{\S\setminus Q} G(x,y)\omega_n(y)dydx,\\
        &= \int_{Q} \omega_n(x) \int_{Q} G(x,y)\omega_n(y)\,dydx + \int_{\S\setminus Q} \omega_n(y)\int_{Q} G(x,y)\omega_n(x)dxdy,\\
        &\leq \int_{Q} \omega_n(x) \int_{Q} G(x,y)\omega_n(y)\,dydx + \int_{\S \setminus Q} \psi_n(x)\omega_n(x)\,dx.
    \end{align*}
    Applying \eqref{ineq:strip concentration of kinetic energy} yields
    \begin{align*}
        \Bigl|E[\omega_n] & - \frac{1}{2}\int_{Q}\int_{Q} G(x,y)\omega_n(x)\omega_n(y)\,dx\,dy \Bigr|  \leq CM^2\left(R^{-\frac{p}{2p-1}} + A^{-\frac{1}{p}} R^{1-\frac{2}{p}} + R^{-2} \right).
    \end{align*}
    By estimating $E[\omega]$ in the same way,
    \begin{align*}
        2|E[\omega_n] - E[\omega]| \le \Bigl| \int_{Q}\int_{Q} G(x,y) \left(\omega(x)\omega(y) - \omega_n(x)\omega_n(y) \right) \,dx\,dy \Bigr| \\
        + CM^2\left(R^{-\frac{p}{2p-1}} + A^{-\frac{1}{p}} R^{1-\frac{2}{p}} + R^{-2} \right).
    \end{align*}

    Next, for any $\eps>0$, there exists a sufficiently large $R$ such that 
    \begin{align*}
        2|E[\omega_n] - E[\omega]| \le \Bigl| \int_{Q}\int_{Q} G(x,y) \left(\omega(x)\omega(y) - \omega_n(x)\omega_n(y) \right) \,dx\,dy \Bigr| + CM^2 A^{-\frac{1}{p}} R^{1-\frac{2}{p}} + \frac{1}{5}\eps,
    \end{align*}
    fix such $n$-independence $R$, there exists another sufficiently large $A$ such that 
    \begin{align*}
        2|E[\omega_n] - E[\omega]| \le \Bigl| \int_{Q}\int_{Q} G(x,y) \left(\omega(x)\omega(y) - \omega_n(x)\omega_n(y) \right) \,dx\,dy \Bigr| + \frac{2}{5}\eps.
    \end{align*}
    Since $Vol(Q)<\infty$, $G(x,y) \in L^{p'}(Q \times Q)$ and $\omega_n(x)\omega_n(y) \rightharpoonup \omega(x)\omega(y)$ in $L^p(Q \times Q)$, for sufficiently large $n$,
    \begin{align*}
        \Bigl| \int_{Q}\int_{Q} G(x,y) \left(\omega(x)\omega(y) - \omega_n(x)\omega_n(y) \right) \,dx\,dy \Bigr| \leq  \frac{1}{5}\eps.
    \end{align*}
    This concludes the convergence of kinetic energy.

\end{proof}

We have already gathered all the ingredients for proving the existence of a minimizer for \eqref{def:minimizing problem}. In the following, we will first define a slightly larger space $\widehat{K}_{\mu}\supseteq K_{\mu}$ and prove the existence of a minimizer in this set, and then strengthen it into a minimizer in $K_{\mu}$.

\begin{prop}\label{prop:strip exists of minimizer}
    Suppose either \ref{case:general convergence on half-plane} or \ref{case:general convergence on strips} holds.
    There exists a compactly supported minimizer $\omega\in K_{\mu}$ that satisfies \eqref{def:symmetry property} to the minimizing problem \eqref{def:minimizing problem}. Consequently, $S_{\mu,\lambda}\neq \emptyset$.
\end{prop}
\begin{proof} 
For any $\mu>0$, define 
\begin{align*}
    \widehat{K}_{\mu}=\left\{ \omega\in L^p(\S) \mid \omega\geq 0, \int_\S x_2\omega dx \leq \mu  \right\},
\end{align*}
and the minimizing problem associated with $\widehat{K}_{\mu}$,
\begin{align}
    \widehat{I}_{\mu,\lambda}=\inf_{\omega\in \widehat{K}_{\mu}}\left\{ -E_{p,\lambda}[\omega] \right\}. \label{def:slightly larger minimizing problem}
\end{align}

\underline{\textit{Step 1. Existence of minimizer $\omega\in \widehat{K}_{\mu}$ to \eqref{def:slightly larger minimizing problem}.}} Let $\{\omega_n\}$ be arbitrary minimizing sequence in $\widehat{K}_{\mu}$ to \eqref{def:slightly larger minimizing problem}. By Steiner symmetrization, we may assume that each $\omega_n$ satisfies \eqref{def:symmetry property}. Since $\{\omega_n\}$ is uniformly bounded in $L^p$ as we proved in Remark \ref{remark:uniformly bounded}, by choosing a subsequence (still denoted by $\{\omega_n\}$), there exists $\omega \in L^p$ such that $\omega_n \rightharpoonup \omega$ in $L^p$, so that $\|\omega\|_p\leq \liminf_{n\to \infty} \|\omega_n\|_p$ and $\|x_2 \omega\|_1\leq \liminf_{n\to \infty} \|x_2 \omega_n\|_1\leq \mu$. The limit $\omega$ belongs to $\widehat{K}_{\mu}$ and satisfies \eqref{def:symmetry property}. Since $\{\omega_n\}$ satisfies the assumption of Proposition \ref{prop:strip energy convergence}, therefore,
\begin{align*}
    -\widehat{I}_{\mu,\lambda} = \lim_{n\rightarrow\infty}E_{p,\lambda}[\omega_n]\leq \lim_{n\rightarrow\infty}E[\omega_n]-\frac{1}{p\lambda}\liminf_{n\rightarrow\infty}\|\omega_n\|_{2}^2 \leq E_{p,\lambda}[\omega]\leq -\widehat{I}_{\mu,\lambda}
\end{align*}

Thus 
\begin{align*}
    -E_{p,\lambda}[\omega] = \widehat{I}_{\mu,\lambda},
\end{align*}
this shows $\omega$ is a minimizer of \ref{def:slightly larger minimizing problem} in $\widehat{K}_{\mu}$. 

\underline{\textit{Step 2. $\omega$ is a minimizer to \eqref{def:minimizing problem} in $K_{\mu}$.}} Let 
\begin{align*}
    \mu_0=\int_{\S} x_2\omega dx,
\end{align*}
then by $\mu_0\leq \mu$ and the subset relation, $\omega$ should be a minimizer to \eqref{def:minimizing problem} with parameters $\mu_0$ and $\lambda$. Thus, according to the structural analysis, Proposition \ref{prop:general structure}, and the analysis of the parameter, Proposition \ref{prop:positive W}, there exists unique $W=W(\omega)>0$ such that 
\begin{align*}
    \omega^{p-1}=\lambda (\psi-Wx_2)_+.
\end{align*}

Further suppose $\mu_0 < \mu$. Notice that $I_{\mu_0,\lambda}=-E_{p,\lambda}[\omega]<0$, then then $\omega$ is a non-trivial function. For any function $h$ chosen in the Proposition \ref{prop:general structure}, there exists a threshold $\eps_0>0$ such that for any $\eps<\eps_0$, 
\begin{align*}
    \int_{\S}x_2(\omega+\eps h)dx=\mu_0+\eps\int_{\S}x_2 h dx \leq \mu,
\end{align*}
thus $\omega+\eps h\in \widehat{K}_\mu$. By maximality of $\omega\in \widehat{K}_{\mu}$, 
\begin{align*}
    0\geq \frac{d}{d\eps}\Bigg|_{\eps=0}E_{p,\lambda}[\omega+\eps h] = \int_{\S} \left( \psi-\frac{1}{\lambda}\omega^{p-1} \right) h dx.
\end{align*}
In the same way as in the proof of Proposition \ref{prop:general structure}, this implies 
\begin{align*}
    \left\{\begin{aligned}
        \psi-\frac{1}{\lambda} \omega^{p-1} & = 0, && \tt{ on } \{\omega>0\}, \\
        \psi & \leq 0, && \tt{ on } \{\omega=0\}.
    \end{aligned}\right.
\end{align*}
Thus $\omega$ satisfy the equation 
\begin{align*}
    \omega^{p-1}=\lambda (\psi- 0 x_2)_+.
\end{align*}
Thanks to the uniqueness of strictly positive $W$, this yields a contradiction. We proved $\mu_0 = \mu$, which means $\omega \in K_{\mu}$, and hence a minimizer to \eqref{def:minimizing problem} in $K_{\mu}$. The compactness of $\spt \omega$ follows from Proposition \ref{prop:compactly support of minimizer}. The proof is complete.

\end{proof}

Finally, given Proposition \ref{prop:strip exists of minimizer}, it is possible to finish the strict subadditivity through the $x_1$-translation invariance of $\S$.

\hfill

\noindent\underline{\textit{Proof of the Proposition \ref{prop:strictly subadditivity} with \ref{case:general convergence on strips}:}}

Let $\omega_1\in K_{\alpha}$ (resp. $\omega_2\in K_{\mu-\alpha}$) be a compactly supported minimizer of $-E_{p,\lambda}$ in $K_{\alpha}$ (resp. $K_{\mu-\alpha}$), then $\omega_1+\omega_2\in K_{\mu}$. By proper $x_1$-translation, it is possible to assume $\spt \omega_1 \cap \spt \omega_2 = \emptyset$, then
\begin{align*}
    I_{\mu,\lambda}
    \leq & -E_{p,\lambda}[\omega_1+\omega_2], \\
    = & - E_{p,\lambda}[\omega_1] - E_{p,\lambda}[\omega_1] - \int_\S \int_\S G(x,y) \omega_1(x) \omega_2(y) dx dy, \\
    < & I_{\alpha,\lambda} + I_{\mu-\alpha,\lambda},
\end{align*}
where the strict inequality in the last step comes from both $\omega_1$ and $\omega_2$ are non-trivial.

\rightline{$\Box$}

\subsection{A General Convergence Theorem} \label{sec:general convergence theorem for half plane and strip}

In this section, we first state the concentration-compactness proposition on the half-plane and the strip (Proposition \ref{prop:concentration-compactness lemma}), and then deduce the first two cases of Theorem \ref{thm:general convergence theorem without L1}.

\begin{prop}[Lemma 4.1 in \cite{abe2022stability} and Proposition 4.13 in \cite{Dong2026StabilityOV}]\label{prop:concentration-compactness lemma}
Let $\O=\R^2_+$ or $\S$, $\mu>0$. Let \(\{\rho_n\}\subset L^1(\O)\) satisfy
\begin{align*}
\rho_n\ge 0,\quad n\ge 1,\qquad \int_{\O}\rho_n\,dx = \mu_n \to \mu \quad\text{as } n\to\infty .
\end{align*}
There exists a subsequence \(\{\rho_{n_k}\}\) satisfying one of the following:
\begin{enumerate}
    \item \textbf{(Compactness)} There exists a sequence \(\{y_k\}\subseteq\overline{\O}\) such that \(\rho_{n_k}(\cdot+y_k)\) is tight, i.e., for every \(\epsilon>0\) there exists \(R>0\) such that
    \begin{align*}
        \liminf_{k\to\infty}\int_{B(y_k,R)\cap \O}\rho_{n_k}\,dx \ge \mu-\epsilon.
    \end{align*}
    If $\O=\S$, it is possible to assume $y_{k,2}=0$ for all $k$.
    \item \textbf{(Vanishing)} For each \(R>0\)
    \begin{align*}
        \lim_{k\to\infty}\sup_{y\in \O}\int_{B(y,R)\cap \O}\rho_{n_k}\,dx = 0 .
    \end{align*}
    \item \textbf{(Dichotomy)} There exists \(\alpha\in(0,\mu)\) such that for every \(\epsilon>0\) there exist \(k_0\ge 1\) and \(\{\rho_k^1\},\{\rho_k^2\}\subset L^1(\O)\) such that \(\spt\rho_k^1\cap\spt\rho_k^2=\emptyset\), \(0\le\rho_k^i\le\rho_{n_k}\) for \(i=1,2\), and
    \begin{align*}
        \limsup_{k\to\infty}\Bigl\{\|\rho_{n_k}-\rho_k^1-\rho_k^2\|_{1}
        + \Bigl|\int_{\O}\rho_k^1\,dx-\alpha\Bigr|
        + \Bigl|\int_{\O}\rho_k^2\,dx-(\mu-\alpha)\Bigr|\Bigr\} \le \epsilon ,
    \end{align*}
    \begin{align*}
    d(\spt\rho_k^1,\spt\rho_k^2) \to \infty \quad\text{as } k\to\infty .
    \end{align*}
\end{enumerate}
\end{prop}

Given that $I_{\mu,\lambda}<0$ and the strict subadditivity of $I_{\mu,\lambda}$ with respect to $\mu$, for any sequence $\{\omega_n\}$ satisfying the assumptions of Theorem \ref{thm:general convergence theorem without L1}, the vanishing and dichotomy cases of the subsequence indicated in Proposition \ref{prop:concentration-compactness lemma} can be excluded respectively. This shows that $\{\omega_n\}$ admits a subsequence satisfy the compactness, thereby completing the proof of Cases \ref{case:general convergence on half-plane} and \ref{case:general convergence on strips} of Theorem \ref{thm:general convergence theorem without L1}.


The following proof is similar to that of Theorem 1.3 in \cite{abe2022stability}, Theorem 5.2 in \cite{abe2025existence}, and Theorem 4.14 in \cite{Dong2026StabilityOV}. Moreover, the case $\Omega=\mathbb{R}^2_+$ with $p=2$ has already been settled in Theorem 2.5 in \cite{Abe2025StabilityOL}. Nevertheless, as a comprehensive and unifying work, we present the complete proof below.

\hspace*{\fill} 

\noindent\underline{\textit{Proof of Case \ref{case:general convergence on half-plane} and \ref{case:general convergence on strips} in Theorem \ref{thm:general convergence theorem without L1}:}}

Let $\{ \omega_n \}$ be a minimizing sequence such that $\omega_n\in K_{\mu_n}$, $\mu_n \to \mu$ and $I_{\mu_n,\lambda} \to I_{\mu,\lambda}$ as $n\to \infty$. By \eqref{ineq:minimizing seq uniformly bounded}, $\{\omega_n\}$ is uniformly bounded in $L^p$. We set $\rho_n = x_2\omega_n$ and apply the Proposition \ref{prop:concentration-compactness lemma}. Then, for
a certain subsequence still denoted by $\{\omega_n\}$, one of the following three cases should occur.

\underline{\textit{Case 1. Dichotomy:}} There exists \(\alpha\in(0,\mu)\) such that for every \(\epsilon>0\) there exist \(k_0\ge 1\) and \(\{x_2\omega_{1,n}\},\{x_2\omega_{2,n}\}\subset L^1(\O)\) such that $\spt\omega_{1,n}\cap\spt\omega_{2,n}=\emptyset$, $0\le\omega_{i,n}\le\omega_{n}$ for $i=1,2$, and $\omega_{3,n}=\omega_{n}-\omega_{1,n}-\omega_{2,n}$ satisfy 
\begin{align*}
    \limsup_{k\to\infty}\Bigl\{\|x_2\omega_{3,n}\|_{1}
    + \Bigl|\alpha_n-\alpha\Bigr|
    + \Bigl|\beta_n-(\mu-\alpha)\Bigr|\Bigr\} \leq \epsilon ,
\end{align*}
where 
\begin{align*}
\alpha_n=\int_{\O}x_2\omega_{1,n} dx, && \beta_n=\int_{\O}x_2\omega_{2,n} dx, 
\end{align*}
and $d_n=d(\spt\omega_{1,n},\spt\omega_{2,n}) \to \infty $ as $n \to \infty$. Choosing a further subsequence, we may assume $\alpha_n \to \alpha_\eps $, $\beta_n\to \beta_\eps$ and $\sup_n \|x_2\omega_{3,n}\|_1\leq 2\eps$. Using estimates \eqref{ineq:estimate of energy crossing term}, \eqref{ineq:estimate of kinetic energy} together with the comparison of Green's function \eqref{ineq:comparison between green's function}, we have
\begin{align*}
    E[\omega_n] = & E[\omega_{1,n}] + E[\omega_{2,n}] + E[\omega_{3,n}] \\
    & + \int_{\O}\int_{\O}G(x,y)(\omega_{1,n}(x)+\omega_{2,n}(x))\omega_{3,n}(y)dxdy + \int_{\O}\int_{\O}G(x,y)\omega_{1,n}(x)\omega_{2,n}(y)dxdy, \\
    \leq & E[\omega_{1,n}] + E[\omega_{2,n}] + C\left(  \|x_2\omega_{3,n}\|_{1}^{\frac{4p-4}{3p-2}}\|\omega_n\|_p^{\frac{2p}{3p-2}} + \mu_n^{\frac{2p-2}{3p-2}}\|x_2\omega_{3,n}\|_p^{\frac{2p}{3p-2}}\right) + \mu_n^2 d_n^{-2}, \\
    \leq & E[\omega_{1,n}] + E[\omega_{2,n}] + C(\mu_n, \mu,\lambda) (\eps^{\frac{4p-4}{3p-2}} + \eps^{\frac{2p}{3p-2}}) + \mu_n^2 d_n^{-2}.
\end{align*}
Thus we obtain
\begin{align*}
    E_{p,\lambda}[\omega_n] \leq & E[\omega_{1,n}] + E[\omega_{2,n}] + C(\mu_n, \mu,\lambda) (\eps^{\frac{4p-4}{3p-2}} + \eps^{\frac{2p}{3p-2}}) + \mu_n^2 d_n^{-2}\\
    & -\frac{1}{p\lambda}\int_{\O}(\omega_{1,n}+\omega_{2,n}+\omega_{3,n})^p dx, \\
    \leq &  E[\omega_{1,n}] + E[\omega_{2,n}] + C(\mu_n, \mu,\lambda) (\eps^{\frac{4p-4}{3p-2}} + \eps^{\frac{2p}{3p-2}}) + \mu_n^2 d_n^{-2} \\
    & -\frac{1}{p\lambda}\int_{\O}\omega_{1,n}^p+\omega_{2,n}^p dx, \\
    = & E_{p,\lambda}[\omega_{1,n}] + E_{p,\lambda}[\omega_{2,n}] + C(\mu_n, \mu,\lambda) (\eps^{\frac{4p-4}{3p-2}} + \eps^{\frac{2p}{3p-2}}) + \mu_n^2 d_n^{-2}.
\end{align*}
Letting $n\to \infty$
\begin{align*}
    -I_{\mu,\lambda} \leq & \limsup_{n\to \infty} E_{p,\lambda}[\omega_{1,n}] + \limsup_{n\to \infty} E_{p,\lambda}[\omega_{2,n}] + C( \mu,\lambda) (\eps^{\frac{4p-4}{3p-2}} + \eps^{\frac{2p}{3p-2}}), \\
    \leq & \limsup_{n\to \infty} (-I_{\alpha_n}) + \limsup_{n\to \infty} (-I_{\beta_n}) + C( \mu,\lambda) (\eps^{\frac{4p-4}{3p-2}} + \eps^{\frac{2p}{3p-2}}), \\
    = & -I_{\alpha_\eps} - I_{\beta_\eps} + C( \mu,\lambda) (\eps^{\frac{4p-4}{3p-2}} + \eps^{\frac{2p}{3p-2}}),
\end{align*}
where the last inequality relies on the continuity of $I_{(\cdot),\lambda}$ established in Proposition \ref{prop:continuous of I_mu}. Recall 
\begin{align*}
\limsup_{k\to\infty}\Bigl\{\Bigl|\alpha_n-\alpha\Bigr|
+ \Bigl|\beta_n-(\mu-\alpha)\Bigr|\Bigr\} \le \epsilon,
\end{align*}
for large $n$, whence
\begin{align*}
\Bigl|\alpha_\eps-\alpha\Bigr|
+ \Bigl|\beta_\eps-(\mu-\alpha)\Bigr| \leq 2\epsilon.
\end{align*}
Letting $\eps\to \infty$, and using the continuity of $I_{(\cdot),\lambda}$,
\begin{align*}
    I_{\mu,\lambda} \geq I_{\alpha} + I_{\mu-\alpha}.
\end{align*}
This contradicts the strict subadditivity \eqref{ineq:strictly subadditivity}. Hence dichotomy cannot occur.

\underline{\textit{Case 2. Vanishing:}} Suppose that vanishing occurs. Namely, for each \(R>0\)
\begin{align*}
\lim_{k\to\infty}\sup_{y\in \O}\int_{B(y,R)\cap \O}x_2\omega_{n}\,dx = 0 .
\end{align*}
We shall prove that $\lim_{n\to \infty}E[\omega_n]=0$ and derive a contradiction to $I_{\mu,\lambda}<0$. For $R>1$ we decompose the energy as
\begin{align*}
    2&E[\omega_n] \\
    & = \left[ \iint\limits_{ \{ |x-y|\geq R\} } + \iint\limits_{ \{ |x-y|< R, G < Rx_2y_2\} } + \iint\limits_{ \{ |x-y|< R, G \geq  Rx_2y_2\} } \right] G(x,y)\omega_n(x)\omega_n(y)dxdy, \\
    & \leq \frac{\mu_n^2}{\pi}R^{-2} + R \iint\limits_{ \{ |x-y|< R\} } x_2y_2\omega_n(x)\omega_n(y)dxdy + \iint\limits_{ \{ |x-y|< R, G \geq Rx_2y_2\} } G(x,y)\omega_n(x)\omega_n(y)dxdy , \\
    & \leq \frac{\mu_n^2}{\pi}R^{-2} + R \mu_n \sup_{y\in \O}\int_{B(y,R)\cap \O}x_2\omega_{n}\,dx + \iint\limits_{ \{ |x-y|< R, G \geq Rx_2y_2\} } G(x,y)\omega_n(x)\omega_n(y)dxdy .
\end{align*}

In the region $|x-y|<R$, the inequality $G_H\geq G \geq R x_2 y_2$ implies $|x-y|\leq R^{-1/2}$. Thus
\begin{align*}
    \iint\limits_{ \{ |x-y|< R, G \geq Rx_2y_2\} } G(x,y)\omega_n(x)\omega_n(y)dxdy \leq \int_{\Omega}\int_{ \{ |x-y|\leq R^{-1/2}\} } G(x,y)\omega_n(y)dy ~ \omega_n(x) dx.
\end{align*}
Choose $r$ with $1<r<p$ and set
$$0<\alpha=\frac{p-r}{r(p-1)} <1,$$ 
so that, after a suitable choice of  $r$, 
$$t=\frac{p-r}{2(r-1)(p-1)}<1,$$
and $|x|^{-2\alpha}$ is integrable near the origin. Let $r'$ be the conjugate exponent of $r$, so that $1/r' = 1/t + 1/r -1$. using $\ln(1+t)\lesssim t^\alpha$ and Young’s convolution inequality we estimate
\begin{align*}
    \int_{\Omega}\int_{ \{ |x-y|\leq R^{-1/2}\} } & G(x,y)\omega_n(y)dy ~ \omega_n(x) dx \\
    & \lesssim \int_{\Omega}\int_{\O} \frac{1}{|x-y|^{2\alpha}} \mathbf{1}_{ B(0,R^{-1/2})}(x-y) \cdot y_2^\alpha\omega_n(y)dy ~ x_2^\alpha \omega_n(x) dx, \\ 
    & \leq \left\| \frac{1}{|x|^{2\alpha}}\mathbf{1}_{ B(0,R^{-1/2})} \ast  x_2^\alpha\omega_n \right\|_{r'} \|x_2^\alpha\omega_n \|_r, \\
    & \leq \left\| \frac{1}{|x|^{2\alpha}}\mathbf{1}_{ B(0,R^{-1/2})}\right\|_{t} \|x_2^\alpha\omega_n \|_r^2, \\
    & \lesssim_p R^{-1/q +\alpha}\|x_2\omega_n\|_1^{2\alpha}\|\omega_n\|_p^{2-2\alpha} \lesssim_{p,\mu_n, \mu, \lambda} R^{-1/q +\alpha}.
\end{align*}
Letting $n\to \infty$ and then $R\to \infty$, 
\begin{align*}
    \lim_{n\to \infty}E[\omega_n] = 0,
\end{align*}
contradicting $I_{\mu,\lambda}<0$. Hence vanishing is also impossible.

\underline{\textit{Case 3. Compactness:}} 
There exists a sequence $\{y_n\} \subset \mathbb{R}^2_+$ such that for arbitrary $\varepsilon > 0$, one can find $R_\eps > 0$ with
\begin{align}
    \liminf_{n\to\infty} \int_{B(y_n,R_\eps)\cap\mathbb{R}^2_+} x_2\omega_n(x) \,dx \ge \mu - \varepsilon. \label{ineq:compactness of minimizing sequence}
\end{align}
After a horizontal translation we may assume $y_n = (0, y_{2,n})$.

\underline{\textit{Case 3.1. $\limsup_{n\to\infty} y_{2,n} = \infty$:}} Choose a subsequence, we can assume $\lim_{n\to\infty} y_{2,n} = \infty$. Define $\psi_n(x)=\int_{\O}G(x,y)\omega_n(y)dy$. Using the supremum estimate \eqref{ineq:sup estimate of stream fct},
\begin{align*}
    2E[\omega_n] & = \left[ \int_{ B(y_n,R_\eps) \cap \O } + \int_{ \O\backslash B(y_n,R_\eps) }\right] \psi_n \omega_n dx, \\
    & \leq \frac{\mu_n\|\psi_n\|_{\infty}}{y_{2,n}-R_\eps} + C \|\omega_n\|_p^{\frac{2p}{3p-2}}\mu_n^{\frac{2p-2}{3p-2}}\|x_2\omega_n\|_{L^1(\O\backslash B(y_n,R_\eps))}^{\frac{2p-2}{3p-2}}, \\
    & \lesssim \frac{\mu_n^{\frac{5p-4}{3p-2}}\|\omega_n\|_p^{\frac{p}{3p-2}}}{y_{2,n}-R_\eps} + \|\omega_n\|_p^{\frac{2p}{3p-2}}\mu_n^{\frac{2p-2}{3p-2}}\eps^{\frac{2p-2}{3p-2}}.
\end{align*}
Hence
\begin{align*}
    0\leq \lim_{n\to\infty} E[\omega_n] 
    \leq C(\mu,\lambda)\eps^{\frac{2p-2}{3p-2}}.
\end{align*}
Since $\eps$ is arbitrary, $ \lim_{n\to\infty} E[\omega_n]=0$. This contradicts to $I_{\mu,\lambda}<0$.

\underline{\textit{Case 3.2. $\limsup_{n\to\infty} y_{2,n} < \infty$:}} Replacing $R_\varepsilon$ by $R(\varepsilon)=R_\varepsilon+\limsup_{n\to\infty}y_{2,n}$, we may assume $y_n=0$. By choosing a further subsequence, there exits $\omega\in L^p(\O)$ such that $\omega_n \rightharpoonup \omega$ in $L^p(\O)$ and
\begin{align}
    \int_{B(0,R)\cap\O} x_2 \omega \,dx \ge \mu - \varepsilon. \label{ineq:compactness of weak limit}
\end{align}
Thus, $\|x_2\omega\|_{1} = \mu$ and $\omega \in K_\mu$.

We now show that the kinetic energy converges under this weak convergence. For any $\eps>0$, pick $R=R(\eps)$ and write $B_+ = B(0,R) \cap \O$,
\begin{align*}
    2E[\omega_n]= \int_{B_+} \psi_n(x) \omega_n(x) \,dx + \int_{\O \setminus B_+} \psi_n(x) \omega_n(x) \,dx.
\end{align*}
Using the symmetry $G(x,y) = G(y,x)$, 
\begin{align*}
    \int_{B_+} \psi_n(x) \omega_n(x) \,dx
    &= \int_{B_+} \omega_n(x) \left( \int_{B_+} G(x,y) \omega_n(y) \,dy + \int_{\O \setminus B_+} G(x,y) \omega_n(y) \,dy \right) dx, \\
    &\leq \int_{B_+} \int_{B_+} G(x,y) \omega_n(x) \omega_n(y) \,dx\,dy + \int_{\O \setminus B_+} \psi_n(x) \omega_n(x) \,dx,
\end{align*}
hence
\begin{align*}
    0 \leq 2E[\omega_n] - \int_{B_+} \int_{B_+} G(x,y) \omega_n(x) \omega_n(y) \,dx\,dy  \leq 2 \int_{\O\setminus B_+} \psi_n(x) \omega_n(x) \,dx.
\end{align*}
By estimate \eqref{ineq:estimate of energy crossing term} and the compactness of minimizing sequence \eqref{ineq:compactness of minimizing sequence}, 
\begin{align*}
    \int_{\O\setminus B_+} \psi_n(x) \omega_n(x) \,dx = \int_{\O} \psi_n(x) \cdot \omega_n(x)\mathbf{1}_{\O\setminus B_+}(x) \,dx \leq C\mu^{\frac{2p-2}{3p-2}}\|\omega_n\|^{\frac{2p}{3p-2}}\eps^{\frac{2p-2}{3p-2}}.
\end{align*}
We therefore obtain
\begin{align*}
    \limsup_{n\to \infty}\left| E[\omega_n] - \int_{B_+} \int_{B_+} G(x,y) \omega_n(x) \omega_n(y) \,dx\,dy \right| \leq C(\mu,\lambda)\eps^{\frac{2p-2}{3p-2}}.
\end{align*}
Similarly, we have
\begin{align*}
    \left| E[\omega] - \int_{B_+} \int_{B_+} G(x,y) \omega(x) \omega(y) \,dx\,dy \right| \leq C(\mu,\lambda)\eps^{\frac{2p-2}{3p-2}}.
\end{align*}
Since $G(x,y) \in L^{p'}(B_+ \times B_+)$ and $\omega_n(x) \omega_n(y) \rightharpoonup \omega(x) \omega(y)$ in $L^p(B_+ \times B_+)$, 
\begin{align*}
    \lim_{n\to\infty}\iint_{B_+\times B_+}G(x,y)\omega_n(x)\omega_n(y)\,dx\,dy =\iint_{B_+\times B_+}G(x,y)\omega(x)\omega(y)\,dx\,dy.
\end{align*}
Combining these facts gives
\begin{align*}
    & \limsup_{n\to \infty} \left| E[\omega] - E[\omega_n] \right| = C(\mu,\lambda)\eps^{\frac{2p-2}{3p-2}}.
\end{align*}
Letting $\eps\to 0$ yields $\lim_{n\to\infty} E[\omega_n] = E[\omega]$. Consequently, 
\begin{align}
    -I_{\mu,\lambda}=\lim_{n\rightarrow\infty}E_{p,\lambda}[\omega_n]\leq \lim_{n\rightarrow\infty}E[\omega_n]-\frac{1}{p\lambda}\liminf_{n\rightarrow\infty}\|\omega_n\|_{p}^p \leq E_{p,\lambda}[\omega]\leq -I_{\mu,\lambda}. \label{eqn:kinetic convergence for half plane and strip}
\end{align}
Thus 
\begin{align*}
    -E_{p,\lambda}[\omega] = I_{\mu,\lambda}, && \|\omega\|_{p}=\lim_{n\rightarrow\infty}\|\omega_n\|_{p},
\end{align*}
the former shows $\omega$ is a minimizer of $I_{\mu,\lambda}=\inf_{\omega\in K_{\mu}}\{-E_{p,\lambda}[\omega]\}$ in $K_{\mu}$, the later show $\omega_n$ converge to $\omega$ strongly in $L^p$. From the tightness estimate \eqref{ineq:compactness of minimizing sequence}, \eqref{ineq:compactness of weak limit}, we have 
\begin{align*}
    \int_{\O} x_2 |\omega_n - \omega| \,dx
    &= \int_{B(0,R)\cap\O} x_2 |\omega_n - \omega| \,dx + \int_{\O \setminus B(0,R)} x_2 |\omega_n - \omega| \,dx, \\
    &\leq \left( \int_{B(0,R)\cap\O}x_2^{p'}dx \right)^{1/p'}\|\omega_n-\omega\|_p + 2\eps.
\end{align*}
Sending $n \to \infty$, $x_2\omega_n \to x_2\omega$ in $L^1(\O)$ follows.

\rightline{$\Box$}

\section{Studies on $\D$.} \label{sec:general convergence theorem for D}

This section focuses on the General Convergence Theorem on $\D$, i.e., Case \ref{case:general convergence on finite volume} of Theorem 1.4. The proof here no longer relies on the concentration-compactness proposition, Proposition \ref{prop:concentration-compactness lemma}; instead, we exploit the decay of $\D$ given by \eqref{def:upper and lower part of domain}, to establish the continuity of the kinetic energy under weak convergence. This continuity yields different outcomes depending on the decay rate. Specifically, when $q\in[0,2]$, the proof can be completed for $p>1$; while for $q>2$, the range will be limited to $1<p \leq \frac{q}{q-1}$. Combining these two regimes with the fact that the boundedness of minimizing sequences holds only for $p>4/3$ (Remark \ref{remark:uniformly bounded}), we obtain precisely the parameter range stated in Case \ref{case:general convergence on finite volume} of Theorem 1.4, and the convergence of the kinetic energy then completes the proof of the General Convergence Theorem.

\begin{prop}[Convergence of the kinetic energy]
\label{prop:weak_finite_volume_energy_convergence}
Let $\O=\D$ and one of the following holds: (1), $q\in[0,2]$, $p>1$; (2), $q\in (2, + \infty)$, $ 1 < p \leq q/(q-1)$.
Suppose \(\{\omega_n\} \subset L^p(\D)\) and \(\omega \in L^p(\D)\), such that 
\begin{align*}
    \omega_n \rightharpoonup \omega  \tt{ in } L^p(\D).
\end{align*}
If each $\omega_n$ is nonnegative, then \(E[\omega_n] \to E[\omega]\).
\end{prop}
\begin{proof}
    \underline{\textit{Case (1):}} It is sufficient to prove that $G\in L^{p^\prime}(\D \times \D)$ for any $1/p+1/p^{\prime}=1$. This follows form:
    \begin{align*}
        & \int_{\D}\int_{\D}G^{p^{\prime}}(x,y)dxdy 
        \\
        \leq & \int_{D_u}\int_{D_u}G_H^{p^{\prime}}(x,y)dxdy + 2\int_{D_l}\int_{D_u}G_H^{p^{\prime}}(x,y)dxdy + \int_{D_l}\int_{D_l}G_H^{p^{\prime}}(x,y)dxdy, \\
        \leq & \int_{D_u}\int_{D_u}G_H^{p^{\prime}}(x,y)dxdy + 3\int_{D_l}\int_{\R^2_+}G_H^{p^{\prime}}(x,y)dxdy, \\
        \leq & \int_{D_u}\int_{D_u}G_H^{p^{\prime}}(x,y)dxdy + C\int_{D_l}y_2^2dy.
    \end{align*}
    The first term is finite because $\tt{Vol}(D_u)<\infty$, and the last term is finite by $q\in[0,2]$ in the sense that 
    \begin{align*}
        \int_{D_l}y_2^2 dy \leq \int_{D_l}y_2^q dy <\infty.
    \end{align*}
    
    \underline{\textit{Case (2):}} We first demonstrate the concentration of $\{x_2\omega_n\}$ in the sense that
    \begin{align}
        \sup_{n}\int_{D_l \cap \{|x_1|>N\}} x_2 \omega_n dx\leq \epsilon. \label{ineq:concentration of x_2 omega_n}
    \end{align}
    Suppose \eqref{ineq:concentration of x_2 omega_n} fails. Then there exists $\epsilon_0>0$ and a subsequence $\{\omega_{n_k}\}$ such that for every $k \in \mathbb{Z}_{\geq 1}$,
    \begin{align*}
        \int_{D_l \cap \{|x_1|>k\}} x_2 \omega_{n_k} dx > \epsilon_0.
    \end{align*}
    Thus 
    \begin{align*}
        \epsilon_0 < \int_{D_l \cap \{|x_1|>k\}} x_2 \omega_{n_k} dx \leq  \left(\int_{D_l \cap \{|x_1|>k\}} x_2^{p^\prime} dx\right)^{1/p^\prime}\left(\int_{D_l \cap \{|x_1|>k\}} \omega_{n_k}^p dx\right)^{1/p}.
    \end{align*}
    where $p^\prime=p/(p-1) $. The assumption $p \leq q/(q-1)$ implies $p^\prime\geq q$, and the weak finite volume condition \eqref{def:weak_finite_volume_domain_assumption} on $D_l$ gives
    \begin{align*}
        \int_{D_l}x_2^{p^\prime}dx\leq \int_{D_l}x_2^{q}dx <+\infty.
    \end{align*}
    Then 
    \begin{align*}
        \int_{D_l \cap \{|x_1|>k\}} x_2^{p^\prime} dx \rightarrow 0, \tt{ as } k \rightarrow \infty,
    \end{align*}
    we conclude
    \begin{align*}
        \|\omega_{n_k}\|_{p} \geq (\int_{D_l \cap \{|x_1|>k\}} \omega_{n_k}^p dx)^{1/p} \geq \epsilon_0 \Big(\int_{D_l \cap \{|x_1|>k\}} x_2^{p^\prime} dx\Big)^{-1/p^\prime}\rightarrow +\infty.
    \end{align*}
    $\{\omega_{n_k}\}$ is uniformly bounded in $L^p$ because $\omega_{n_k}\rightharpoonup\omega$ in $L^p(D)$, a contradiction arises. 
    To complete the proof, fix $0<\epsilon<1$ and define
    $$
        D_0=D_u\cup(D_l \cap \{|x_1|\leq N\}),
    $$
    where $N=N(\epsilon^{\frac{3p-2}{2p-2}})$ is the constant obtained in the previous step. By properly increasing $N$, if needed, we can assume
    \begin{align*}
        \int_{\D\backslash D_0} x_2\omega dx\leq \epsilon^{\frac{3p-2}{2p-2}}.
    \end{align*}
    Then 
    \begin{align*}
        |E[\omega_n] & - E[\omega]| \leq |E[\omega_n|_{D_0}]-E[\omega|_{D_0}]| \\
        & + \int_{\mathcal{D}}\int_{\mathcal{D}} G(x,y)\omega_n|_{D_0}(x)\omega_n|_{\mathcal{D}\backslash D_0}(y)dydx + E[\omega_n|_{\mathcal{D}\backslash D_0}]\\
        & + \int_{\mathcal{D}}\int_{\mathcal{D}} G(x,y)\omega|_{D_0}(x)\omega|_{\mathcal{D}\backslash D_0}(y)dydx + E[\omega|_{\mathcal{D}\backslash D_0}].
    \end{align*}

    For the last four terms, using \eqref{ineq:estimate of kinetic energy}, \eqref{ineq:estimate of energy crossing term} together with $$\sup_n\|x_2\omega_n|_{\mathcal{D}\backslash D_0}\|_1, \|x_2 \omega|_{\mathcal{D}\backslash D_0}\|_1\leq \epsilon^{\frac{3p-2}{2p-2}},$$ 
    we obtain
    \begin{align*}
        |E[\omega_n] & - E[\omega]| \leq |E[\omega_n|_{D_0}]-E[\omega|_{D_0}]| + C M^2\epsilon.
    \end{align*}
    where $M$ is a large constant such that $\sup_n\|\omega_n\|_{p}, \|\omega\|_{p}\leq M$.

    For the first term, note that $\tt{Vol}(D_0)<\infty$, so $G(x,y) \in L^{p^{\prime}}(D_0 \times D_0)$ with $1/p+1/{p^\prime}=1$ and $\omega_n(x)\omega_n(y) \rightharpoonup \omega(x)\omega(y)$ in $L^p(D_0 \times D_0)$. Hence, as $n\to\infty$
    \begin{align*}
        E[\omega_n|_{D_0}] = & \int_{D_0}\int_{D_0} {G}(x,y)\omega_n(x)\omega_n(y)dydx \\
        & \longrightarrow \int_{D_0}\int_{D_0} {G}(x,y)\omega(x)\omega(y)dydx = E[\omega|_{D_0}].
    \end{align*}

    In conclusion, as $n$ is sufficiently large, 
    \begin{align*}
        |E[\omega_n] & - E[\omega]| \leq 2CM^2\epsilon.
    \end{align*}
    Since $\epsilon>0$ can be arbitrary small, the proof is complete.
    
\end{proof}

Based on Proposition \ref{prop:weak_finite_volume_energy_convergence}, we can then obtain the convergence theorem when $\O=\D$.

\hspace*{\fill} 

\noindent\underline{\textit{Proof Case \ref{case:general convergence on finite volume} in Theorem \ref{thm:general convergence theorem without L1}:}}

The proof follows the same lines as that of the first two cases of Theorem 1.4 in Section \ref{sec:general convergence theorem for half plane and strip}; the only difference is that the convergence of the kinetic energy is guaranteed by Proposition \ref{prop:weak_finite_volume_energy_convergence} rather than by the concentration‑compactness argument of Proposition \ref{prop:concentration-compactness lemma}. We therefore omit the details.

\rightline{$\Box$}

\section{Stability respect to the minimizer}\label{sec:stability}

Having established the existence of minimizers and a general convergence theorem, we now turn to their orbital stability for initial data in $C_c^\infty(\Omega)$. Using the general convergence theorem (Theorem 1.4), we show that the set of minimizers is orbitally stable. The proof here is quite similar with the Theorem 1.4 in \cite{abe2022stability}, the Theorem 1.3 in \cite{Abe2025StabilityOL}, and the Theorem 5.1 in \cite{Dong2026StabilityOV}, but to ensure completeness, we present it below.

\hspace*{\fill} 


\noindent\textit{\underline{Proof of the Theorem \ref{thm:stable wrt minimizer}: }}
Suppose \eqref{eqn:stable result} false, then there exists $\epsilon_0>0$, a sequence $\{\xi_{0,n}\}\subseteq C_c^\infty(\O), ~ \xi_{0,n} \geq 0, ~ \|\xi_{0,n}\|_1\leq 1 $ and $t_n$ in the lifespan of the solution $\xi_n$ based on the initial data $\xi_{0,n}$ such that 
\begin{align*}
    \inf_{\omega\in S_{\mu,\lambda}}\Big\{ \|\xi_{0,n}-\omega\|_{p}+\|x_2(\xi_{0,n}-\omega)\|_{1} \Big\} & \leq 1/n, \\ \inf_{\omega\in S_{\mu,\lambda}}\Big\{ \|\xi_n(t_n)-\omega\|_{p}+\|x_2(\xi_n(t_n)-\omega)\|_{1} \Big\} &> \epsilon_0.
\end{align*}
We write $\xi_n=\xi_n(t_n)$ by suppressing $t_n$. We take $\omega_n\in  S_{\mu,\lambda}$ such that $\|\xi_{0,n}-\omega\|_{p}+\|x_2(\xi_{0,n}-\omega)\|_{1}\rightarrow 0$. By \eqref{ineq:estimate of energy difference}, 
$$ | E_{p,\lambda}[\xi_{0,n}]+I_{\mu,\lambda} | = | E_{p,\lambda}[\xi_{0,n}]-E_{p,\lambda}[\omega_n] | \rightarrow 0 \text{ as } n \rightarrow \infty$$
Thus $\{\xi_{0,n}\}$ be a minimizing sequence such that $\xi_{0,n}\in K_{\mu_n}, \mu_n\rightarrow\mu$ and $-E_{p,\lambda}[\xi_{0,n}] \rightarrow I_{\mu,\lambda} $ as $n\rightarrow\infty$.

Knowing that $\xi$ conserved on particle trajectory map, then $\|\xi_n\|_{p}=\|\xi_{0,n}\|_{p}$ and $\|x_2\xi_n\|_1=\|x_2\xi_{0,n}\|_1$. Meanwhile, since the kinetic energy is conserved in time, then $E_{p,\lambda}[\xi_{0,n}]=E_{p,\lambda}[\xi_{n}]$ for any $n$. Hence $\{\xi_{n}\}$ be a minimizing sequence such that $\xi_{n}\in K_{\mu_n}, \mu_n\rightarrow\mu$ and $-E_{p,\lambda}[\xi_{n}] \rightarrow I_{\mu,\lambda} $ as $n\rightarrow\infty$. By the Theorem \ref{thm:general convergence theorem without L1}, choosing a subsequence still denoted by $\{\xi_{n}\}$, there exists $\xi\in S_{\mu,\lambda}$ such that $\omega_n\rightarrow\omega$ and $x_2\omega_n\rightarrow x_2\omega$ in $L^2(\Omega)$ and $L^1(\Omega)$ respectively. Sending $n\rightarrow\infty$, 
\begin{align*}
    0=&\inf_{\omega\in S_{\mu,\lambda}}\Big\{ \|\xi-\omega\|_{p}+\|x_2(\xi-\omega)\|_{1} \Big\},\\
    =&\inf_{\omega\in S_{\mu,\lambda}}\Big\{ \lim_{n\rightarrow\infty}\Big( \|\xi_{n}-\omega\|_{p}+\|x_2(\xi_{n}-\omega)\|_{1}\Big) \Big\},\\
    \geq&\liminf_{n\rightarrow\infty}\Big(\inf_{\omega\in S_{\mu,\lambda}}\Big\{ \|\xi_{n}-\omega\|_{p}+\|x_2(\xi_{n}-\omega)\|_{1}\Big\} \Big) 
    \geq \epsilon_0.
\end{align*}
We obtained a contradiction.

\rightline{$\Box$}

After completing the proof of Theorem \ref{thm:stable wrt minimizer}, we note that the coverage of the last case ($\Omega = D$) in that theorem is limited. To remedy this, we introduce an upper bound on the $L^1$-norm into the admissible space, thereby obtaining stability with an $L^1$-norm constraint for all $q\ge 0$ and $p>1$. Since this discussion deviates from the main line of reasoning, we place it in Appendix \ref{sec:stability for D with L1 bound}.


\section{Discussion}\label{sec:discussion}

In the end, we discuss some generalizations of Theorem \ref{thm:stable wrt minimizer}, and state some open questions. 

\subsection{More general domains and optimal conditions} 

Our treatment of weak finite volume domains exploits the decay rate $q$ of the lower part $D_l$ to enforce concentration, which leads to the constraint $p \le q/(q-1)$ in the absence of an $L^1$ bound. Under the same decay rate but with an $L^1$ bound, however, there is no upper restriction on $p$. This naturally raises the question of whether such a restriction is merely technical or genuinely necessary.  

Conversely, it would be of considerable interest to identify optimal geometric conditions that guarantee $L^p$ stability for $p>4/3$ (or even $p>1$) within the concentration‑compactness framework used by Abe, Choi, and Jeong \cite{Abe2025StabilityOL} for the half‑plane. 

\subsection{Free‑boundary problems.}

As observed in \cite{Dong2026StabilityOV}, the weak finite volume condition appears naturally in the study of free‑boundary Euler equations (see e.g. Hu, Luo, and Yao \cite{hu2024small}). The stability results obtained here concern fixed domains; a challenging next step is to understand whether an analogous variational principle can be formulated for the evolving fluid domains that occur in free‑boundary problems.

If the domain itself is part of the solution, the Biot-Savart law, which serves as the foundation of this paper, breaks down.  This leads to even greater difficulties. One possible approach is to decompose the velocity field $u$ into a rotational part $u_{\rm rot}$ and an irrotational part $u_{\rm irrot}$.
For a fixed time $t$, $u^{rot}$ can be handled in a manner similar to that developed in this paper, while $u^{irrot}$ could be treated as an additional contribution. By exploiting the boundary equation, the extra accumulated terms might then be reduced to a negligible perturbation, thereby yielding the desired result.


\subsection{Multiple patches and finer structure of minimizers.}

Our analysis guarantees the orbital stability of the set of minimizers, but it says very little about the structure of individual minimizers except for the relation $\omega^{p-1} = \lambda(\psi - W x_2)_+$. For the special case $p=2$, Abe and Choi \cite{abe2022stability} proved uniqueness (up to translation) of the Chaplygin–Lamb dipole, and Abe, Choi, and Jeong \cite{Abe2025StabilityOL} removed the mass constraint from that result. For $p \neq 2$ on the half‑plane, uniqueness remains open even under Steiner symmetry. The problem becomes even more intriguing on strips and weak finite volume domains. A detailed geometric study of the set $S_{\mu,\lambda}(p)$, for instance, reveal the existence of disconnected supports or asymmetric touching configurations. Progress in this direction might rely on the maximum principle techniques of Fraenkel \cite{fraenkel2000introduction} and the symmetry results used in \cite{abe2022stability} for the Lamb dipole, adapted to the nonlinearity $t^{1/(p-1)}_+$.

\subsection{Connections with kinetic energy maximization.}

In \cite{abe2025existence}, Abe, Choi, Jeong, Sim, and Woo construct a family of Sadovskii vortices for all $p\in(1,\infty]$ by maximizing the kinetic energy subject to constraints on impulse and $L^p$ norm. Their approach yields axis‑touching configurations share the same structure as the minimizers found here, i.e. $\omega^{p-1} = \lambda(\psi - W x_2)_+$ with $W>0$. A systematic comparison between the two variational principles
could deepen our understanding of the selection mechanisms for traveling wave profiles. In particular, a natural question is whether there exists a homotopy or a similar relation between the minimizers and maximizers of these two types of problems. An affirmative answer would help establish a rigorous equivalence between them.

\subsection*{Data availability}
No data was used for the research described in the article.

\appendix

\section{Stability for $\D$ with $L^1$ bound}\label{sec:stability for D with L1 bound}

In this section we study the stability on $\D$ under a suitable $L^1$ bound. The idea of the proof is adapt from the Section \ref{sec:general convergence theorem for D}, but after introducing the $L^1$-norm constraint, the decay property of the domain $\D$ itself plays a stronger role. 

We first establish a mixed $L^1$–$L^p$ upper bound for the kinetic energy, and then use it to obtain the $L^p$ boundedness of minimizing sequences for the new minimization problem \eqref{def:minimizing problem with L1} that will be defined below. Based on this boundedness and the decay property of $\D$, we prove a general convergence theorem, Theorem \ref{thm:general convergence for D with L1}, which is valid for all $q\ge 0$ and $p>1$. This theorem then yields the corresponding orbital stability. Furthermore, to align with the preceding structural analysis carried out in Proposition \ref{prop:general structure}, we also study the structure of the minimizer in this setting.

\begin{prop}\label{prop:estimate for kinetic energy with L1}
    The following estimates hold for $\omega,\omega_i\in L^1\cap L^p(\D)$ provided that $x_2\omega$ and $x_2\omega_i$ are both in $L^1(\D)$, with a constant $C$ independent from $\omega,\omega_i, i=1,2$:
    \begin{align}
        &E[\omega]\leq C\|\omega\|_p^{1/2}\|\omega\|_1^{1/2+1/p}\|x_2\omega\|_1^{1-1/p},  \label{ineq:estimate of kinetic energy with L1}
        \\
        &\left|\int_{\D}\int_{\D}G(x,y)\omega_1(y)\omega_2(x)dxdy\right|\leq C\|\omega_1\|_p^{1/2}\|\omega_1\|_1^{1/2}\|x_2\omega_2\|_1^{1-1/p}\|\omega_2\|_1^{1/p}, \label{ineq:estimate of energy crossing term with L1}
        \\
        &|E[\omega_1]-E[\omega_2]|\leq C \|\omega_1-\omega_2\|_{p}^{1/2}\|\omega_1-\omega_2\|_{1}^{1/2}\|x_2(\omega_1+\omega_2)\|_{1}^{1-1/p}\|\omega_1+\omega_2\|_{1}^{1/p}. \label{ineq:estimate of energy difference with L1}
    \end{align}
\end{prop}
\noindent\textit{Proof: } Using \eqref{ineq:comparison between green's function} and Hölder’s inequality with the decomposition $1=\frac{1}{2}\cdot (1-\frac{1}{p})  + \frac{1}{2} \cdot \frac{1}{p} + \frac{1}{2}$, we have
\begin{align*}
    \left|\int_{\D}G(x,y)\omega(y)dy\right|\leq &  
    \left(\int_{\D}G(x,y)^qdy\right)^{1/2(1-1/p)}\|\omega\|_p^{1/2}\|\omega\|_1^{1/2}, \\
    \leq & \left(\int_{\R^2_+}G_H(x,y)^qdy\right)^{1/2(1-1/p)}\|\omega\|_p^{1/2}\|\omega\|_1^{1/2}, \\
    \leq & Cx_2^{1-1/p}\|\omega\|_p^{1/2}\|\omega\|_1^{1/2},
\end{align*}
where $q$ is suitably chosen. Now,
\begin{align*}
    \left|\int_{\D}\int_{\D}G(x,y)\omega_1(x)\omega_2(y)dxdy\right|
    \leq & C\|\omega_1\|_p^{1/2}\|\omega_1\|_1^{1/2}\|x_2^{1-1/p}\omega_2\|_1,\\
    \leq & C\|\omega_1\|_p^{1/2}\|\omega_1\|_1^{1/2}\|x_2\omega_2\|_1^{1-1/p}\|\omega_2\|_1^{1/p},
\end{align*}
which is \eqref{ineq:estimate of energy crossing term with L1}.
The estimate \eqref{ineq:estimate of kinetic energy with L1} follows immediately by taking $\omega_1=\omega_2=\omega$. As for \eqref{ineq:estimate of energy difference with L1}, set $\widetilde\omega=\omega_1-\omega_2, \widehat\omega=\omega_1+\omega_2$; then
\begin{align*}
    2(E[\omega_1]-E[\omega_2]) 
    = & \int_{\D}\int_{\D}G(x,y)\widetilde\omega(x)\widehat\omega(y)dxdy,
\end{align*}
and \eqref{ineq:estimate of energy difference with L1} follows from \eqref{ineq:estimate of energy crossing term with L1}.

\rightline{$\Box$}

Later, for parameters $0\le\mu<\infty$ and $\nu\in\R_+$ we introduce the admissible set
\begin{align*}
    \widetilde{K}_{\mu,\nu}=\left\{\omega\in L^p(\D) \big| \omega\geq0, \int_{\D}x_2\omega(x)dx=\mu, \|\omega\|_1\leq \nu \right\},
\end{align*}
together with the new minimization problem
\begin{align}
    \widetilde{I}_{\mu, \nu, \lambda}(p)=\inf_{\omega \in \widetilde{K}_{\mu,\nu}}\{-E_{p,\lambda}\}. \label{def:minimizing problem with L1}
\end{align}
As usual, $\widetilde{S}_{\mu, \nu, \lambda}(p) \subseteq \widetilde{K}_{\mu,\nu}$ denotes the set of minimizers of \eqref{def:minimizing problem with L1}. Although the non‑homogeneity of $E[\omega]$ and $\|\omega\|_p^p$ makes a spatial scaling argument failed, the parameter $\nu$ is not important during the discussion. We therefore abbreviate the notation as 
\begin{align*}
    \widetilde{K}_{\mu}=\widetilde{K}_{\mu,\nu}(p), && 
    \widetilde{I}_{\mu,\lambda}=\widetilde{I}_{\mu, \nu,\lambda}(p), && 
    \widetilde{S}_{\mu, \nu, \lambda} = \widetilde{S}_{\mu, \lambda}(p).
\end{align*}

Using above notations, we have

\begin{prop}\label{prop:propoties of I_mu with L1} 
    \begin{align}
        & \widetilde{I}_{0,\lambda}=0, && \text{ for any $\lambda\in \R$},   \label{eqn:I_0 with L1}
        \\
        & \widetilde{I}_{\mu,\lambda}>-\infty, && \text{ for any } \mu\geq 0 \text{ and } \lambda>0 ,\label{eqn:I_mu is finite with L1}\\
        & \widetilde{I}_{\mu,\lambda}<0, && \text{ for any } (\mu,\lambda)\in \bigcup_{K\subseteq \Omega \text{ and } |K|<\infty}
         (0, \mu_K ] \times (\lambda_{K,\mu},\infty ). \label{ineq:negativity of I_mu with L1}
    \end{align}
    where
    \begin{align*}
        \mu_K=\frac{\nu}{\tt{Vol(K)}^{1/p}} \int_K x_2 dx, && \lambda_{K,\mu}=\frac{2\mu^{p-2}\tt{Vol(K)}}{p \int_K\int_K G(x,y) dydx}\left(\int_K x_2\right)^{2-p}
    \end{align*}
\end{prop}
\noindent\textit{Proof: } Property \eqref{eqn:I_0 with L1} is trivial since $K_0=\{0\}$. By \eqref{ineq:estimate of kinetic energy with L1} and Young's inequality, for any $\omega\in \widetilde{K}_{\mu}$,
\begin{align*}
    E_{p,\lambda}[\omega]&\leq C\|\omega\|_p^{1/2}\|\omega\|_1^{1/2+1/p}\|x_2\omega\|_1^{1-1/p}-\frac{1}{p\lambda}\|\omega\|_p^p 
    \leq C \lambda^{\frac{1}{2p-1}}\mu^{\frac{2p-2}{2p-1}}\|\omega\|_{1}^{\frac{p+1}{2p-1}},
\end{align*}
where we choose $\eps=(C p\lambda)^{-\frac{1}{2p}}$. Hence
\begin{align*}
    \widetilde{I}_{\mu,\lambda}=\inf_{\omega\in \widetilde{K}_{\mu}}\big\{-E_{p,\lambda}[\omega ]\big\}\geq -C \lambda^{\frac{1}{2p-1}}\mu^{\frac{2p-2}{2p-1}}\|\omega\|_{1}^{\frac{p+1}{2p-1}}>-\infty.
\end{align*}

For any $(\mu,\lambda)\in (0,\mu_K ] \times (\lambda_{K,\mu},\infty)$ with $K \subseteq \Omega \text{ and } |K|<\infty$, define
\begin{align*}
    \omega_0=c_0 \mathbf{1}_{K}, \text{ where } c_0=\mu\left(\int_{K}x_2dx\right)^{-1}.
\end{align*}
Then $\omega_0\in \widetilde{K}_{\mu} $ by the choice of $c_0$, and 
\begin{align*}
    E_{p,\lambda}[\omega_0]&= \frac{1}{2}c_0^2\int_{K}\int_{K}G(x,y)dxdy-\frac{1}{p\lambda}c_0^p \tt{Vol}(K)\\
    &=\frac{1}{2}c_0^2\left(\int_{K}\int_{K}G(x,y)dxdy-\frac{2c_0^{p-2}}{p\lambda}\tt{Vol}(K)\right)\\
    &>0,
\end{align*}
where the last inequality follows thanks to the choice of $\lambda$.

\rightline{$\Box$}

\begin{remark}[Minimizing Sequence is $L^p$ Bounded]\label{remark:uniformly bounded with L1}
\end{remark}
    Any minimizing sequence $\{\omega_n\}$ satisfying $\omega_n\in \widetilde{K}_{\mu}, \mu_n\rightarrow \mu$, and $-E_{p,\lambda}[\omega_n]\rightarrow \widetilde{I}_{\mu,\lambda}$ with $\mu, \lambda$ follows the assumption in \eqref{ineq:negativity of I_mu with L1} is uniformly bounded in $L^p$. Indeed, by \eqref{ineq:estimate of kinetic energy with L1} and Young’s inequality, for any $\epsilon >0 $ and $\omega\in K_\mu$,
    \begin{align*}
        \frac{1}{p\lambda}\|\omega\|_{p}^p+E_{p,\lambda}[\omega]
        = & E[\omega] \leq C\left(\eps^{2p}\|\omega\|_p^p + \eps^{-\frac{2p}{2p-1}}\|\omega\|_1^{\frac{p+2}{2p-1}}\|x_2\omega\|_1^{\frac{2p-2}{2p-1}}\right)
    \end{align*}
    Choosing $\eps=(2Cp\lambda)^{-\frac{1}{2p}}$ gives, 
    \begin{align*}
        \|\omega\|_{p}^p\leq C\lambda^{\frac{2p}{2p-1}}\|\omega\|_1^{\frac{p+2}{2p-1}}\|x_2\omega\|_1^{\frac{2p-2}{2p-1}}-2p\lambda E_{p,\lambda}[\omega].
    \end{align*}
    Since $\widetilde{I}_{\mu,\lambda}<0$, the minimizing sequence satisfies 
    \begin{align*}
        \limsup_{n\rightarrow +\infty}\|\omega_n\|_{p}\leq C\lambda^{\frac{2}{2p-1}}\mu^{\frac{2p-2}{2p^2-p}}\limsup_{n\rightarrow +\infty}\|\omega_n\|_1^{\frac{p+2}{2p^2-p}}. 
    \end{align*}
    In particular, if $\omega$ is a minimizer, then 
    \begin{align}
        \|\omega\|_{p}\leq C\lambda^{\frac{2}{2p-1}}\mu^{\frac{2p-2}{2p^2-p}}\|\omega\|_1^{\frac{p+2}{2p^2-p}}. \label{ineq:uniformly bounded with L1}
    \end{align}


\begin{prop}[Convergence of the kinetic energy]
\label{prop:weak_finite_volume_energy_convergence_with_L^1}
Let $\O=\D$. Assume \(\{\omega_n\} \subset L^1  \cap L^p(\D)\) and \(\omega \in L^p(\D)\) satisfy:
\begin{itemize}
    \item \(\sup_n \|\omega_n\|_1 + \sup_n \|\omega_n\|_p\le M\) for some constants \(M\), and
    \item \(\omega_n \rightharpoonup \omega\) weakly in \(L^p(\mathcal{D})\).
\end{itemize}
If each $\omega_n$ is nonnegative, then for any $\epsilon>0$, there exists $N=N(\epsilon)>0$, such that 
\begin{align}
    \sup_{n}\int_{D_l \cap \{|x_1|>N\}} x_2 \omega_n dx\leq \epsilon, \label{ineq:concentration of bdd seq with L1}
\end{align}
hence 
\(E[\omega_n] \to E[\omega]\).
\end{prop}
\begin{proof}
Observe that, if $\D$ meet the weak finite condition with parameter $q_0<q$, then 
\begin{align*}
    \int_{D_l} x_2^q dx \leq \int_{D_l} x_2^{q_0} dy <\infty,
\end{align*}
so that $\D$ meet the weak finite condition with parameter $q$. Hence, with such a subcase relation, it is sufficient to assume $q \geq p/(p-1)$. 

Suppose \eqref{ineq:concentration of bdd seq with L1} is fails. Then there exists $\epsilon_0>0$ and a subsequence $\{\omega_{n_k}\}$ such that for any $k \in \mathbb{Z}_{\geq 1}$,
\begin{align*}
    \int_{D_l \cap \{|x_1|>k\}} x_2 \omega_{n_k} dx > \epsilon_0.
\end{align*}
Thus 
\begin{align*}
    \epsilon_0 < \int_{D_l \cap \{|x_1|>k\}} x_2 \omega_{n_k} dx \leq  \left(\int_{D_l \cap \{|x_1|>k\}} x_2^{q} dx\right)^{1/{q}}\left(\int_{D_l \cap \{|x_1|>k\}} \omega_{n_k}^{q'} dx\right)^{1/{q'}}.
\end{align*}
where $\frac{1}{q}+\frac{1}{q'}=1$ with $q'= q/(q-1) \leq p $. By the weak finite volume condition \eqref{def:weak_finite_volume_domain_assumption} for $D_l$,
\begin{align*}
    \int_{D_l \cap \{|x_1|>k\}} x_2^q dx \rightarrow 0, \tt{ as } k \rightarrow \infty.
\end{align*}
Then 
\begin{align*}
    \|\omega_{n_k}\|_{q'} \geq (\int_{D_l \cap \{|x_1|>k\}} \omega_{n_k}^{q'} dx)^{1/q'} \geq \epsilon_0 \Big(\int_{D_l \cap \{|x_1|>k\}} x_2^q dx\Big)^{-1/q}\rightarrow +\infty.
\end{align*}

However, this leads to a contradiction: our assumption \(\|\omega_{n_k}\|_1 + \|\omega_{n_k}\|_p \le M\) implies \(\|\omega_{n_k}\|_{q'} \le 2M\) for every \(q' \in (1,2]\). The proof of \(E[\omega_n] \to E[\omega]\) follows the same lines as the second case in Proposition \ref{prop:weak_finite_volume_energy_convergence}, and using inequalities \eqref{ineq:estimate of energy crossing term with L1} and \eqref{ineq:estimate of kinetic energy with L1} to control the four tailing terms; we omit the details.

\end{proof}

Proposition  \ref{prop:weak_finite_volume_energy_convergence_with_L^1} plays essentially the same role as Proposition \ref{prop:weak_finite_volume_energy_convergence}. With it, we therefore conclude the following general convergence theorem.

\begin{thm}\label{thm:general convergence for D with L1}
    Let $p>1$, $\nu>0$, and $(\mu,\lambda)$ satisfy \eqref{ineq:negativity of I_mu with L1}. For any minimizing sequence $\{\omega_n\}$ satisfying $\omega_n\in \widetilde{K}_{\mu_n}$, $\mu_n \rightarrow \mu $ and $-E_{p,\lambda}[\omega_n] \rightarrow \widetilde{I}_{\mu,\lambda}$, then there exists a subsequence, still denoted by $\{\omega_n\}$, such that there exists $\omega\in \widetilde{K}_{\mu}$, 
    $\omega_n\rightarrow\omega$ and $x_2\omega_n\rightarrow x_2\omega$ strongly in $L^p(\mathcal{D})$ and $L^1(\mathcal{D})$ respectively. In particular $\omega\in \widetilde{S}_{\mu,\lambda}$ and hence $\widetilde{S}_{\mu,\lambda}\neq \emptyset$.
\end{thm}
\begin{proof}
    The argument is the same as that for the last case of Theorem \ref{thm:general convergence theorem without L1}, using Proposition \ref{prop:weak_finite_volume_energy_convergence_with_L^1} instead of Proposition \ref{prop:weak_finite_volume_energy_convergence}; we therefore omit it. 

\end{proof}

\hspace*{\fill} 

\noindent\textit{\underline{Proof of the Theorem \ref{thm:stable wrt minimizer with L1 bound}: }}

The proof is identical to that of Theorem \ref{thm:stable wrt minimizer}, except that it relies on the general convergence theorem in Theorem \ref{thm:general convergence for D with L1} rather than Theorem \ref{thm:general convergence theorem without L1}.  We omit the repetition.
    
\rightline{$\Box$}

To conclude this appendix and the paper, we give the structural analysis of a minimizer $\omega \in \widetilde{S}_{\mu,\lambda}$.

\begin{prop}[General structure of minimizers]
\label{prop:general structure with L1}
Let $p>1$, $\nu>0$, and $(\mu,\lambda)$ satisfy \eqref{ineq:negativity of I_mu with L1}. Each minimizer $\omega \in \widetilde{S}_{\mu,\lambda}$ satisfies
\begin{equation}
\label{eqn:general structure of minimizer with L1}
\omega^{p-1} = \lambda(\psi -W x_{2} - \gamma)_+, \qquad 
\psi (x) = \int_{\D}G(x,y)\omega (y)\mathrm{d}y,
\end{equation}
for some constants $W, \gamma \in \R$, uniquely determined by $\omega$.
\end{prop}
\noindent\textit{Proof: }
Take an arbitrary minimizer $\omega \in \widetilde{S}_{\mu,\lambda}$ . Since $\widetilde{I}_{\mu,\lambda}< 0$ by \eqref{ineq:negativity of I_mu with L1}, $\omega\not\equiv 0$ . There exists a $\delta_{0} > 0$ such that $Vol(\{x\in \R^2_+\mid \omega \geq \delta_{0}\}) > 0$. Fix compactly supported $h_{i}\in L^{\infty}(\Omega)$ such that $\spt h_{i}\subset \{\omega \geq \delta_{0}\}, i = 1,2,$ and
\begin{align*}
    \int_{\D}h_{1}(x)\mathrm{d}x = 1, && \int_{\D}x_{2}h_{1}(x)\mathrm{d}x = 0, && 
    \int_{\D}h_{2}(x)\mathrm{d}x = 0, && \int_{\D}x_{2}h_{2}(x)\mathrm{d}x = 1.
\end{align*}

For any $\delta \in (0, \delta_{0})$ and compactly supported $h \in L^{\infty}(\D)$ such that $h \geq 0$ on $\{0 \leq \omega \leq \delta \}$, define
\begin{align*}
    \eta = h - \left(\int_{\Omega}h\mathrm{d}x\right)h_{1} - \left(\int_{\Omega}x_{2}h\mathrm{d}x\right)h_{2},
\end{align*}
so that $\int \eta dx = \int x_{2}\eta dx = 0$. For sufficiently small $\eps>0$, $\omega+\epsilon\eta\geq \delta-\epsilon\|\eta\|_{\infty}\geq 0$ on $\{\omega > \delta \}$. Since $\eta = h\geq 0$ on $\{0\leq \omega\leq \delta\}$, we also have $\omega+\epsilon\eta\geq 0 $ on $\{0\leq \omega\leq \delta\}$. Hence $\omega+\epsilon\eta\in \widetilde{K}_\mu$ for small $\eps>0$. Since $\omega$ is a minimizer of \eqref{def:minimizing problem with L1}, then
\begin{align*}
    0\geq \frac{d}{d\eps}\Bigg|_{\eps=0}E_{p,\lambda}[\omega+\eps\eta] 
    = \int_{\Omega}\left(\psi-Wx_2-\gamma-\frac{1}{\lambda}\omega^{p-1}\right)h = \int_{\omega>\delta}+\int_{0< \omega \leq \delta}\left( \Psi-\frac{1}{\lambda}\omega^{p-1} \right)h,
\end{align*}
where we use $\Psi=\psi-W x_2 -\gamma$ and
\begin{align*}
    \gamma=& \int_{\Omega}\left(\psi-\frac{1}{\lambda}\omega^{p-1}\right)h_1, && W=\int_{\Omega}\left(\psi-\frac{1}{\lambda}\omega^{p-1}\right)h_2.
\end{align*}

In the same way as in the proof of Proposition \ref{prop:general structure}, this implies 
\begin{align}
    \left\{\begin{aligned}
    \Psi - \frac{1}{\lambda}\omega^{p-1} & = 0 \quad \text{on } \{\omega > 0\}, \\
    \Psi & \leq 0 \quad \text{on } \{ \omega =0\}.
    \end{aligned}\right.
    \label{eqn:Psi and omega passing delta to 0 with L1}
\end{align}

Thus $\omega^{p-1} = \lambda(\psi -W x_{2} - \gamma)_+$. The uniqueness of $W,\gamma$ follows by taking $h=\pm h_1,\pm h_2$ as in the proof of Proposition \ref{prop:general structure}. 

\rightline{$\Box$}

\bibliographystyle{plain}
\bibliography{sample}

\end{document}